\documentclass[11pt, leqno]{amsart}

\usepackage{color}
\usepackage{amsmath,amssymb,amsfonts,amsthm,bm}
\usepackage{graphicx}
\usepackage{float}
\usepackage{enumerate}
\usepackage{appendix}
  \usepackage[shortlabels]{enumitem}
\usepackage[numbers, sort&compress]{natbib}
 \usepackage[left=1.35in, right=1.35in]{geometry}
\def\R{\mathbb R}
\def\Z{\mathbb Z}

\def\p{\mathbb P}
\def\q{\mathbb Q}

\def\N{\mathbb N}

\def\cC{\mathcal C}

\def\d{\partial}
\def\ep{\epsilon}
\def\t{\dot}

\def\B{\dot{B}}
\def\H{\dot{H}}
\def\W{\dot{W}}

\renewcommand\div{{\rm div}\,}
\renewcommand\lim{{\rm lim}\,}
\renewcommand\exp{{\rm exp}\,}
\renewcommand\sup{{\rm sup}\,}
\renewcommand\inf{{\rm inf}\,}
\renewcommand\log{{\rm log}\,}
\renewcommand\det{{\rm det}\,}

\newcommand{\with}{\quad\hbox{with}\quad}
\newcommand{\andf}{\quad\hbox{and}\quad}

\newcommand{\Int}{\displaystyle \int}

\newtheorem{theorem}{Theorem}[section]
 \newtheorem{corollary}[theorem]{Corollary}
 
 \newtheorem{proposition}[theorem]{Proposition}
 \theoremstyle{definition}
 \newtheorem{definition}[theorem]{Definition}
 \theoremstyle{remark}
 \newtheorem{remark}[theorem]{Remark}
 
 \numberwithin{equation}{section}
\newcommand{\bmu}{\bar{\mu}}
\newcommand{\bk}{\bar{k}}
\newcommand{\wh}{\widehat}
\newcommand{\wt}{\widetilde}
\newcommand{\eps}{\varepsilon}

\newcommand{\abs}[1]{\left\vert#1\right\vert}

\newcommand{\norm}[1]{\Vert#1\Vert}

\begin{document}
\title[The density-dependent incompressible Navier-Stokes-Korteweg system ]
{Global existence  and uniqueness of the density-dependent incompressible Navier-Stokes-Korteweg system with variable capillarity and viscosity coefficients}
%\author{ Rapha\"el Danchin}
\author{Shan Wang}
\begin{abstract}
We consider the global well-posedness of the inhomogeneous incompressible Navier-Stokes-Korteweg system with a general capillary term. Based on the maximal regularity property, we obtain the global existence and uniqueness of solutions to the incompressible Navier-Stokes-Korteweg system with variable viscosity and capillary terms. By assuming the initial density  $\rho_0$  is close to a positive constant, additionally,
the initial velocity $u_0$ and the initial density  $\nabla \rho_0$ are small 
in critical space  $\B^{-1+d/p}_{p,1}(\R^{d})$ $(1<p<d).$  This work relies on the maximal regularity property of the heat equation, of the Stokes equation, and of the  Lam\'e equation.

\end{abstract}
\date{}

\maketitle
\section*{Introduction}
In this paper, our aim is to extend to more elaborated models what is known for  the inhomogeneous incompressible Navier-Stokes system: \begin{equation*}
\left\{\begin{aligned}
&\d_t\rho+\div(\rho u)=0, \\
 & \rho \d_t u+\rho u\cdot \nabla u-\mu \Delta u+\nabla P=0, \\
&\div u=0,\\
&(\rho,u)|_{t=0}=(\rho_{0},u_{0}).
\end{aligned}\right.\eqno(INS)
\end{equation*}
One can consider for instance the case where the internal capillarity is taken into account.
This leads to the so-called  Navier-Stokes-Korteweg system that has been derived at the end of the  XIX-th century, 
and has known a renewed interest recently to model fluids with phase transition. 
For this system, the density is expected to be smoothed out, so that 
the interface between the phases of the fluid under consideration are diffuse, 
as opposed to sharp interfaces for the classical modelization based on $(INS).$
This presents a decisive advantage in the numerical study 
of fluids with phase transition (the numerical  aspect 
will not be considered in this thesis, though). 

\medbreak
A classical  model to study the dynamics of a fluid endowed with internal capillarity is the following general compressible Korteweg system:

 \begin{equation*}
\left\{\begin{aligned}
&\d_t\rho+\div(\rho u)=0, \\
 &\d_t(\rho u)+\div(\rho u\otimes u)-\mathcal{A} u+\nabla(P(\rho))=\div K,
\end{aligned}\right.
\end{equation*}
where $\rho(t,x)\in \R$ and $u(t,x)\in \R^d,$ respectively, denote the density and velocity of the fluid at position $(t,x)\in \R_+\times \R^d$ with $d\geq 2.$ The diffusion operator is ${\mathcal A}u=2\div(\mu(\rho)D(u))+\nabla(\lambda(\rho)\div u),$ and the differential operator $D$ is defined by $D(u)=\nabla u+Du$ with $(D(u))_{i,j}=\d_iu^j+\d_ju^i.$

The viscosity coefficients $\mu$ and $\lambda$ are  smooth functions 
of $\rho$ satisfying $\mu>0$ and $\lambda+2\mu>0,$
 the pressure $P$ a given function  which is usually chosen as the  Van der Waals pressure:
 $$P(\rho, T^*)=\frac{RT^*\rho}{b-\rho}-a\rho^2$$
 where $R$ is the specific gas constant, $a,$ $b$ are positive constants, 
and the general capillary tensor $K$ satisfies  
\begin{equation}\label{eq:divK}\div K=\nabla \Bigl(\rho k(\rho)\Delta \rho+\frac 12(k(\rho)+\rho k'(\rho))\abs{\nabla \rho}^2\Bigr)-\div (k(\rho)\nabla \rho \otimes \nabla \rho)\end{equation}
with $k$ a smooth function of $\rho$ (a common choice is $k(\rho)=\kappa/\rho$ with $\kappa>0$). 
 \medbreak
 In this paper, we are concerned with the following \emph{incompressible} Navier-Stokes-Korteweg system: 
 %\begin{color}{red} I erased $\bar\mu$ and $\bar k$ in the system to simplify the presentation. These coefficients do no play any role in your work. I also removed the definition of $\rho,$ $P$ and so on
% below the system, since the notations have been already introduced just
% before.  \end{color}
 \begin{equation*}
\left\{\begin{aligned}
&\d_t\rho+\div(\rho u)=0, \\
 & \rho \d_t u+\rho u\cdot \nabla u- \div(\mu(\rho)D(u))+\nabla P=-\div(k(\rho)\nabla \rho\otimes \nabla \rho), \\
&\div u=0,\\
&(\rho,u)|_{t=0}=(\rho_{0},u_{0}).
\end{aligned}\right.\leqno(INSK)
\end{equation*}
where $\rho(t,x)\in \R$, $P(t,x)\in \R$, $u(t,x)\in \R^d$ respectively denote the unknown density, pressure, and velocity of the fluid at the position $(t,x)\in \R_+\times \R^d$ with $d=2,3$.  The scalar functions $\mu$ and $k$ are smooth function $\R \to \R.$ 
%We will focus on solutions around the stable equilibrium state $(\bar \rho,0),$ and in order to simplify the notation we will %choose $\bar \rho=1.$  
%\medbreak
%Note that, being  gradients, both the  second viscous term, namely $\nabla(\lambda(\rho)\div u),$  and the first term of the right-hand side of \eqref{eq:divK} may be put  in the pressure term.
%\medbreak

So far, very few mathematical works have been devoted to the study of this incompressible version of the Navier-Stokes-Korteweg system. Although the existence of strong solutions
on a small time interval is known (see the results mentioned below), the all-time existence
is an open question, even for small data. 

The reason why it is much more difficult to construct global solutions than 
in the compressible case is that, as can be seen from the definition of the capillary pressure in \eqref{eq:divK},
the highest order part of the capillary term is a gradient and thus acts as a modified pressure, without providing any dissipation to the density while, in the compressible case, it smooths out the density. 
In fact, when studying $(INSK),$ one has  to handle the highly nonlinear term
$\div(k(\rho)\nabla \rho\otimes \nabla \rho)$ without 
the benefit of any regularizing or dissipative properties of the density.
Furthermore, owing to 
the incompressibility condition, there is no part of velocity that can be combined with the density fluctuation so as to glean some regularization effect.

% In addition, the capillary term can be rewritten as a general divergence %form (from \eqref{eq:reru1},  the difference  of $\rho \nabla \Delta \rho$ %and $-\nabla \rho \Delta \rho$ is a gradient, adsorbed by the pressure term).

 %Any energy method with use of symmetrizers is completely useless. That is %why this system is much less studied than the compressible version where %the presence of capillarity does regularize the density.

\smallbreak
As in every (reasonable) physical model, one can derive a formal energy functional associated with $(INSK). $
Indeed, 
%denoting the inner product in $L_2$  by $(f|g):=\int_{\R^d} f g\,dx,$ 
%it is clear that \begin{equation}    \label{eq:reru1}
 %   ( \nabla \rho \cdot \Delta \rho|u)=( u\cdot \nabla \rho | \Delta \rho)=-( \d_t \rho | \Delta \rho)=\frac12\frac d{dt}\|\nabla%\rho\|_{L_2}^2. \end{equation}
%Let us shortly explain how to exhibit  the energy balance for the system (INSK). 
using the mass equation  and integrating by parts several times we find that
\begin{equation}\label{eq:rhe}
    \frac{1}{2}\frac{d}{dt} \int_{\R^d} k(\rho)\abs{\nabla \rho}^2\,dx=\int_{\R^d} \div(k(\rho)\nabla \rho \otimes \nabla \rho)\cdot u\,dx.
\end{equation}
Then, testing the momentum equation with $u$, we obtain 
\begin{equation*}
\frac{d}{dt}\left(\int_{\R^d} \rho \abs{u}^2\,dx+\int_{\R^d} k(\rho)\abs{\nabla \rho}^2\,dx\right)+ \int_{\R^d} \mu(\rho)\abs{D (u)}^2\,dx=0.
\end{equation*}
In the end, denoting  $$E(t):=\int_{\R^d} \rho(t) \abs{u(t)}^2\,dx+ \int_{\R^d} k(\rho(t))\abs{\nabla \rho(t)}^2\,dx$$
and integrating \eqref{eq:rhe} on $[0,t]$ gives:
\begin{equation}\label{eq:enes}
E(t)+\int_0^t \!\!\int_{\R^d} \mu(\rho(\tau))\abs{\nabla u(\tau)}^2\,dxd\tau=E(0).
\end{equation}
%\begin{color}{red} Maybe you can talk about the energy balance here
%\end{color}
It is not known whether the above energy balance can be used to construct a global 
weak solutions (as in the noncapillary case). 
However, it is possible to use similar ideas to control locally-in-time some Sobolev norms
of the solution. 
For example,  J.~Yang, L.~Yao, and C.~Zhu \cite{YYZ} obtained local unique solutions as well as convergence toward the Euler system as the capillarity and viscosity coefficients go to zero, by assuming   $(\rho_0-\bar \rho,u_0)\in H^4\times H^3$,  for the above system with capillarity term $\kappa \rho \nabla \Delta \rho.$

Lately, T.~Wang \cite{TW} solved the initial boundary problem of the system $(INSK)$ and proved that it  admits a unique local strong solution (with, possibly, vacuum) whenever the initial data $(\rho_0, u_0) \in \W^{2,p}\times (H^1_0\cap H^2)$ satisfies  $\div u_0=0$ and the compatibility condition:
$$-\div(\mu(\rho_0)D(u_0))+\nabla P_0+ \div(k(\rho_0)\nabla \rho_0\otimes \nabla \rho_0 )=\rho_0^{\frac 12}g$$
for some $(P_0,g)\in H^1\times L_2.$
\smallbreak
Recently, A. Jabour and A. Bouidi proved that  if one only considers $u_0$ in $H^1_0,$
then the local existence and uniqueness of strong solution to the initial boundary problem of the system $(INSK)$ in \cite{TW} still holds true 
\emph{without the above compatibility condition.} These articles strongly rely on the $L_2$ structure of the system that
helps to handle the capillary term.
\smallbreak
By other techniques, in  \cite{BGLRS}, the authors provide local solutions in  bounded domains and for  initial data in Sobolev spaces $W^{s}_{p} (p\geq 1)$ which are also in  $L_2$  for the following system:
\begin{equation*}
\left\{\begin{aligned}
&\d_t\rho+\div(\rho u)=0, \\
 &\d_t(\rho u)+\div(\rho u\otimes u)-\mu \Delta u+\nabla P=-\kappa \nabla \rho \Delta \rho \\
&\div u=0,\\
&(\rho,u)|_{t=0}=(\rho_{0},u_{0}).
\end{aligned}\right.
\end{equation*}
%without adding (artificial)  diffusion to the mass equation.
It is based on the following relation:
$$\rho \nabla \Delta \rho=\nabla (\rho \Delta \rho)-\nabla \cdot (\nabla \rho \otimes \nabla \rho)+\frac{1}{2}\nabla (\abs{\nabla \rho}^2)=\nabla (\rho \Delta \rho)-\nabla \rho \Delta \rho$$
that allows to include the term $\nabla (\rho \Delta \rho)$ in the pressure term.
\medbreak
Another way to deal with the capillary term is to consider the system $ (INSK)$ in Lagrangian coordinates \cite{BC, CB2017}, after adapting the approach 
that has been  developed in \cite{RD2014,DM2012} for $(INS), $
or in \cite{BC,DM1} for the compressible Navier-Stokes equations.
The advantage of using Lagrangian coordinates is that  
the advection term of the velocity disappears and that,
in the incompressible situation, the density becomes independent of time. Therefore, if we apply the Lagrangian method to the system $(INSK),$ then the corresponding capillary term only depends on the initial density. In \cite{BC}, C.~Burtea and 
F.~Charve obtained local well-posedness in critical Besov space based on Lagrangian change variables, and initial data satisfying:
$$u_0\in \B^{-1+\frac{d}{p}}_{p,1}\!\!\with\!\! \div u_0=0,\quad\rho_0-1,\nabla \rho_0\in \B^{\frac{d}{p}}_{p,1}
\!\!\andf\!\! \underset{x\in \R^d}{\inf}\rho_0(x)>0.$$ Furthermore,  if $\bk:=k(1)$  and $\bmu:=\mu(1),$ then 
they find that the life span $T$ of a smooth solution 
satisfies $T\geq {C\bmu}/{\bk}$ if the initial velocity is small. 
\smallbreak
Here, we also mention the works of Haspot (see \cite{BH1,BH2,BH4, BH}) where the \emph{compressible} Korteweg
system is studied in general settings, or with minimal assumptions on the initial data
(in critical  Besov spaces) for suitable viscosity and capillarity coefficients.
\medbreak
Very recently, the author tried to use the maximal regularity property and regularity estimates for the transport equation to 
consider the system $(INSK).$ In \cite{TSW}, the author obtained the local well-posedness result for data with slightly subcritical regularity and a lower bound for the lifespan of the solutions.
\smallbreak
 In this paper, we improve the proofs in the author's thesis \cite{TSW} to prove the global well-posedness.  By using  the similar maximal regularity property for the heat equation, Stokes equation, and Lam\'e equation in Lorentz spaces as in \cite{DW} for $(INS)$ 
 so as to solve global-in-time $(INSK)$ with general smooth viscosity and capillary coefficients, with small initial data. 
 
 The rest of the paper unfolds as follows. In the next section, we introduce the tools and introduce our main results, additionally, we also show the main idea to deal with the system $(INSK)$. Section \ref{section2} is devoted to the prior estimates of Theorem \ref{them1nsk}.  The proof of the existence and uniqueness part of Theorem \ref{them1nsk} will be established in Section \ref{section3}.

 %Note that %we here consider  small initial velocity for conciseness, but that easy 
 %modifications of the proof would allow us to consider large velocities. 
% Furthermore,
%in contrast with the aforementioned works, our techniques are not based at all on Fourier analysis, and 
% are thus adaptable to the case where the fluid domain is a smooth bounded open set for example. 

 \medbreak

%\section{The  incompressible Navier-Stokes-Korteweg system }
%\numberwithin{equation}{section}

\section{Tools and results}

Before showing the tools and main results for $(INSK),$ introducing a few notations and recalling some results is in order. 
\medbreak
First, throughout the text,
$A\lesssim B$ means that  $A\leq CB$, where $C$ designates 
 various positive real numbers the value of which does not matter. 
 \smallbreak
For any Banach space $X,$ index $q$ in $[1,\infty]$  and time $T\in[0,\infty],$ we use the notation 
$\|z\|_{L_q(0,T;X)}\overset{\text{def}}{=}
\bigl\| \|z\|_{X}\bigr\|_{L_q(0,T)}.$
If $T=\infty$, we define a similar way the space  $\|z\|_{L_q(\R_+;X)}$ (also
denoted by $\|z\|_{L_q(X)}$). Notation $\|z_1, z_2\|_{L_q(\R_+;X)}:=\|z_1\|_{L_q(\R_+;X)}+\|z_2\|_{L_q(\R_+;X)}. $ In the following, the space $X=X(\R^d),$ if we haven't particular point out. 
%If $T=\infty$, then we  just  write 
In the case where $z$ has $n$ components $z_k$ in $X,$ we 
 keep the notation  $\norm{z}_X$ to 
mean $\sum_{k\in\{1,\cdots,n\}} \norm{z_k}_X$. 
\smallbreak
We shall use the following notation for the \emph{convective derivative}: 
\begin{equation*}
\frac{D}{Dt}\overset{\text{def}}{=}\d_{t}+u\cdot \nabla \andf \t{u}\overset{\text{def}}{=}u_{t}+u\cdot \nabla u.\end{equation*}
 Next, let us recall the definition of Besov spaces on $\R^d.$ 
Following \cite[Chap. 2]{BCD}, we fix two smooth functions $\chi$ and $\varphi$ 
 such that 
$$\displaylines{\text{Supp}\  \varphi \subset \{\xi\in\R^{d},\: 3/4\leq \abs{\xi}\leq 8/3\}\ \ \text{and  } \ \ \forall\xi\in\R^{d}\setminus\{0\}, \;\underset{j\in\Z}{\sum}\varphi(2^{-j}\xi)=1,\cr
\text{Supp}\ \chi\subset\{\xi\in \R^{d},\: \abs{\xi}\leq 4/3\} \ \text{and}\ \ \forall \xi\in \R^{d},\;  \chi(\xi)+\sum_{j\geq 0}\varphi(2^{-j}\xi)=1,}$$
and set for all $j\in\Z$ and  tempered distribution $u,$ 
 $$\dot{\Delta}_{j}u\overset{\text{def}}{=}\mathcal{F}^{-1}(\varphi(2^{-j}\cdot)\wh{u})\overset{\text{def}}{=}2^{jd} \wt{h}(2^{j}\cdot)\star u\with\wt{h}\overset{\text{def}}{=}\mathcal{F}^{-1}\varphi,$$
\begin{equation*}\dot{S}_{j}u\overset{\text{def}}{=}\mathcal{F}^{-1}(\chi(2^{-j}\cdot)\wh{u})\overset{\text{def}}{=}2^{jd}h(2^{j}\cdot)\star u\with h \overset{\text{def}}{=}\mathcal{F}^{-1}\chi,\end{equation*}
where $\mathcal{F}u$ and $\wh{u}$ denote the Fourier transform of $u.$ 

\begin{definition}[Homogeneous Besov spaces]
Let $(p,r)\in [1,\infty]^{2}$ and $s\in\R.$ 
We set 
$$\norm{u}_{\B^{s}_{p,r}(\R^{d})}\overset{\text{def}}{=}\norm{(2^{js}\norm{\dot{\Delta}_{j} u}_{L_{p}(\R^{d})})_{j\in \Z}}_{\mathit{\ell}_{r}(\Z)}$$
and  denote by $\B^{s}_{p,r}(\R^{d})$ the set of tempered distributions $u$ on $\R^d$
 such that  $\norm{u}_{\B^{s}_{p,r}(\R^{d})}<\infty$ and
 \begin{equation}\label{eq:lf}
 \underset{j\to -\infty}{\lim}\norm{\dot S_{j}u}_{L_{\infty}(\R^{d})}=0.
 \end{equation}
 \end{definition}
%It is classical that the scaling invariance
%condition for $u_0$ pointed out in \eqref{isi} 
%is satisfied for all elements of 
%$\dot B^{-1+d/p}_{p,r}(\R^d)$ with $1\leq %p,r\leq\infty.$
 \medbreak
Next, we define  Lorentz spaces,   and recall a useful characterization. 
\begin{definition}
Given $f$ a measurable function on a measure space $(X,\nu)$ and $1\leq p,r\leq \infty$, we define 
$$\wt\|{f}\|_{L_{p,r}(X,\nu)}:=
\begin{cases}
(\int_{0}^{\infty}(t^{\frac{1}{p}}f^{*}(t))^{r}\,\frac{dt}{t})^{\frac1r} & \text{if $r<\infty$},\\
\underset{t>0}{\sup} t^{\frac{1}{p}}f^{*}(t)& \text{if $r=\infty$},
\end{cases}$$
where $$f^{*}(t):=\inf\bigl\{s\geq 0:\nu(\{\abs{f}>s\})\leq t\bigr\}\cdotp$$
The set of all $f$ with $\wt\|{f}\|_{L_{p,r}(X,\nu)}<\infty$ is called the Lorentz space with indices $p$ and $r$.
\end{definition}
\begin{remark}\label{lorentzdef2}
The Lorentz space $L_{p,p}(X,\nu)$ coincides with 
the Lebesgue space $L_p(X,\nu).$
In addition, according to \cite[Prop.1.4.9] {LG}, the Lorentz spaces may be endowed with the following (equivalent) quasi-norm:  
\begin{equation*}
\norm{f}_{L_{p,r}(X,\nu)}:=
   \begin{cases}
   p^{\frac{1}{r}}\biggl(\Int_{0}^{\infty}\bigl(s\,\nu(\{\abs{f}>s\})^{\frac{1}{p}}\bigr)^{r}\,\frac{ds}{s}\biggr)^{\frac{1}{r}} & \text{if $r<\infty$}\\
  \underset{s>0}{\sup} s\,\nu(\{\abs{f}>s\})^{\frac{1}{p}}& \text{if $r=\infty$}.
   \end{cases}
\end{equation*}
\end{remark}

The results will strongly rely on  a maximal regularity property  for the following evolutionary Stokes system:
  \begin{equation}\label{eq:stokes}
\left\{\begin{aligned}
 &u_{t}-\mu \Delta u+\nabla P=f &\ \ \text{in }\ \R_+\times \R^{d}, \\
 &\div u=0  &\ \ \text{in }\ \R_+\times \R^{d}, \\
&u|_{t=0}=u_{0} &\ \ \text{in }\  \R^{d},
\end{aligned}\right.
\end{equation}
 and 
 the following Lam\'e system:
\begin{equation}
\label{sys:lema}
\left\{\begin{aligned}
 &  \d_t u-\mathcal{A}u=f&\ \ \text{in }\ \R_+\times \R^{d}, \\
&u|_{t=0}=u_0 &\ \ \text{in }\ \R^{d}, \\
\end{aligned}\right.
\end{equation}
with Lam\'e operator $\mathcal{A}u=\mu\Delta u+\mu'\nabla \div u $ with $\mu>0,$ $\sigma=\mu+\mu'>0.$  

We notice that the above two systems can be rewritten as heat equation and it has been pointed out in  \cite[Prop. 2.1]{DM2} 
 that for the free heat equation,
 if assume that initial data $u_0$ in 
 $\dot B^{2-2/q}_{p,r}(\R^d),$
 then the solution $u$ is such that $u_t$ and $\nabla^2u$ are in $L_{q,r}(\R_+;L_p(\R^d)),$ $u$ in 
 $\dot B^{2-2/q}_{p,r}(\R^d).$ 
 
 To be convenient, we introduce the following function space:
  \begin{equation}\label{eq:W}
  \W^{2,1}_{p,(q,r)}(\R_+\times \R^{d}):=\bigl\{u\in \mathcal{C}(\R_+;\B^{2-2/q}_{p,r}( \R^{d})):u_{t}, \nabla^{2}u\in L_{q,r}(\R_+;L_{p}( \R^{d})) \bigr\}\cdotp\end{equation}
  The maximal regularity property of the heat equation in  \cite[Prop. 2.1]{DM2}  motivate us to present the following property of system \eqref{sys:lema} and  Proposition \ref{propregularity} of system \eqref{eq:stokes} which plays a very important role in Section \ref{section2}. 
\smallbreak
\begin{proposition}
    \label{pro:mrforlame}
    Let $1<p,q< \infty$ and $1\leq r\leq \infty.$ Then, for any $u_{0}\in \B^{2-2/q}_{p,r}(\R^{d})$ and any $f\in L_{q,r}(\R_+;L_{p}(\R^{d})),$ the system
\eqref{sys:lema}
has a unique solution $u$ with 
and\footnote{Only weak continuity holds if $r=\infty.$}  
$u$ in the space $\W^{2,1}_{p,(q,r)}(\R_+;\times \R^{d}).$ 
Furthermore, there exists a constant
$C$ such that
\begin{multline}\label{eq:maxreg11}
\sigma_1^{1-1/q}\norm{u}_{L_{\infty}(\R_+;\B^{2-2/q}_{p,r}(\R^{d}))}+\norm{u_{t}, \sigma_1\nabla^{2}u}_{L_{q,r}(\R_+;L_{p}(\R^{d}))}
\\
\leq C\bigl(\sigma_2^{1-1/q}\norm{u_{0}}_{\B^{2-2/q}_{p,r}(\R^{d})}+\norm{f}_{L_{q,r}(\R_+;L_{p}(\R^{d}))}\bigr)\end{multline}
with $\sigma_1=\min \{\mu, \sigma\}$ and $\sigma_1=\max \{\mu, \sigma\}.$
Then,  if $2/q+d/p>2,$ we have for all $(s,m)$ such that 
$$\frac{d}{2m}+\frac{1}{s}=\frac{d}{2p}+\frac{1}{q}-1 \with s\in(q,\infty)\andf  m\in(p,\infty),$$ it holds that 
\begin{multline}\label{eq:maxreg31}\sigma_1^{1+\frac{1}{s}-\frac{1}{q}}\norm{u}_{L_{s,r}(\R_+;L_{m}(\R^{d}))}\\\leq C\bigl(\sigma_1^{1-1/q}\norm{u}_{L_{\infty}(\R_+;\B^{2-2/q}_{p,r}(\R^{d}))}+\norm{u_{t}, \sigma_1\nabla^{2}u}_{L_{q,r}(\R_+;L_{p}(\R^{d}))}\bigr)\cdotp
\end{multline}
\end{proposition}
\begin{proof}
Let $\p$ and $\q$ be the Helmholtz projectors defined as
\begin{equation}\label{eq:defpq}\p:={\rm Id}+\nabla (-\Delta)^{-1}\div\andf \q:=-\nabla (-\Delta)^{-1}\div.
\end{equation}
 After applying $\p$ and $\q$ to system \eqref{sys:lema}, it becomes the heat equation of $\p u$ and $\q u$:
   \begin{equation*}
\left\{\begin{aligned}
 &  \d_t (\p u)-\mu\Delta (\p u)=\p f&\ \ \text{in }\ \R_+\times \R^{d}, \\
  &  \d_t (\q u)-\sigma (\q u)=\q f&\ \ \text{in }\ \R_+\times \R^{d}, \\
&(\p u,\q u)|_{t=0}=(\p u_0,\q u_0) &\ \ \text{in }\ \R^{d}, \\
\end{aligned}\right.
\end{equation*}
Then, according to Proposition 2.1 in  \cite{DM2} and using that the projectors $\p,$ $\q$  are continuous on $L_{q,r}(\R_+;L_{p}(\R^{d}))$
gives \eqref{eq:maxreg11} and \eqref{eq:maxreg31} for $\p u$ and $\q u$ with index $\mu$ and $\sigma$ respectively. Note that $u=\p u+\q u$ implies \eqref{eq:maxreg11} and \eqref{eq:maxreg31} for $ u.$
\end{proof}
%As in \cite{BC}, we consider the critical case $2-2/q=-1+d/p.$ Furthermore, 
%for reasons that will be explained later on
% (in particular the fact that $\dot B^{d/p}_{p,r}(\R^d)$
% embeds in $L_\infty(\R^d)$ if and only if $r=1$), 
% we shall only consider Besov spaces of type 
% $\dot B^{d/p-1}_{p,1}(\R^d).$
\medbreak
The main result of this paper is the following global well-posedness result for data with critical regularity.
\begin{theorem}\label{them1nsk} Let $d=2,3.$
    Assume that $1<p,q<d$ with $d/2p+1/q=3/2.$ Let   $s\in(q,\infty),$ $m\in(p,\infty)$ such that 
$$\frac{d}{2m}+\frac{1}{s}=\frac{d}{2p}+\frac{1}{q}-1.$$
Assume that the initial divergence-free velocity  $u_{0}$ is in $\B^{-1+d/p}_{p,1}(\R^{d}),$  
  and that $\rho_{0}$ belongs to $L_{\infty}(\R^{d}),$ $\nabla \rho_{0}$ is in $\B^{-1+d/p}_{p,1}(\R^{d}).$  There exists a constant $c>0$ such that if 
\begin{equation}\label{inidr}
    \norm{\rho_{0}-1}_{L_{\infty}(\R^{d})}, \norm{\nabla \rho_{0}}_{\B^{-1+d/p}_{p,1}(\R^d)}, \norm{u_{0}}_{\B^{-1+d/p}_{p,1}(\R^d)}< c, 
\end{equation}
 then $(INSK)$ has a unique global-in-time  solution 
 $(\rho,u,\nabla P)$  satisfying,
 $u, \nabla \rho\in \dot W^{2,1}_{p,(q,1)}(\R_+\times\R^d),$ $ \nabla P\in L_{q,1}(\R_+;L_p(\R^d)),$
   \begin{equation*}\norm{\rho-1}_{L_{\infty}(\R_+\times\R^{d})}
 =    \norm{\rho_{0}-1}_{L_{\infty}(\R^{d})}< c,
 \end{equation*}
 and  the following properties: 
 \begin{itemize}
 \item     $\nabla u\in L_1(\R_+;L_\infty(\R^d))$ and $u\in L_2(\R_+; L_\infty(\R^d))$;
 \item $tu\in L_\infty(\R_+;\dot B^{1+d/m}_{m,1}(\R^d))$;
 \item $\bigl(u,(tu)_t, \nabla^2(tu),\nabla(tP)\bigr)
 \in L_{s,1}(\R_+;L_m(\R^d))$;
 \item $t^{1/2}u \in L_\infty (\R_+\times \R^d);$
  \item     $\nabla^2 \rho\in L_1(\R_+;L_\infty(\R^d))$ and $\nabla \rho\in L_2(\R_+; L_\infty(\R^d))$;
 \item $t\nabla \rho\in L_\infty(\R_+;\dot B^{1+d/m}_{m,1}(\R^d))$;
 \item $\bigl(\nabla \rho,(t\nabla \rho)_t, \nabla^2(t\nabla \rho)\bigr)
 \in L_{s,1}(\R_+;L_m(\R^d))$;
 \item $t^{1/2}\nabla \rho \in L_\infty (\R_+\times \R^d).$
 \end{itemize}
 \end{theorem}
  % \begin{theorem}
  % \label{them2}
  % \end{theorem}
  % \begin{color}{red} I removed $\nabla\rho\in \cap L_{\wt p}$ as 
 %  it follows from $L_p\cap L_\infty.$\end{color}
   %\begin{remark}
%Recall that in the compressible situation, to handle the capillary term
%it is common practice to use the change of unknown
%\begin{equation*}
%        v=u+\nabla \log \rho.
%      \end{equation*}
%The advantage is that after this transformation, 
%the equation of density becomes a heat equation, namely
 %    $$\rho_t-\Delta \rho=-\div(\rho v).$$
%Consequently, it is possible to combine maximal regularity estimates
%both for the density and the (modified) velocity so as to prove  optimal well-posedness statements. 
%In our situation, performing the above change of unknown does not seem to be very helpful since $v$ satisfies a Stokes system with a
%non-divergence free constraint depending on $%\nabla\rho.$
% \end{remark}
% \begin{color}{red}
 %\begin{remark}
% Let us recall that in \cite{BC} C. Burtea and F. Charve
%proved a local well-posedness statement for different Besov spaces that was based on maximal $L_1$ regularity. Compared to this work, our approach does not rely on Fourier analysis so that it could be extended to more general domains than $\R^d$ 
%quite easily, a question that we plan to address in the near future. 
 %\end{remark}
% \end{color}
Proving  Theorem \ref{them1nsk} will be mostly based on % standard regularity estimates for the transport equation (satisfied by the density), standard regularity estimates for the Heat equation (satisfied by the new velocity), and
the maximal regularity result of Proposition \ref{propregularity} 
to estimate the velocity and density.  In order to achieve the  Lipschitz property of the velocity $\nabla u \in L_1(L_\infty)$ and for reasons that will be explained later on 
(in particular the fact that $\dot B^{d/p}_{p,r}(\R^d)$
 embeds in $L_\infty(\R^d)$ and $\dot B^{-1+d/p}_{p,r}(\R^d)$
 embeds in $L_d(\R^d)$ if and only if $r=1$),  
we will need to work at the critical one $\W^{2,1}_{p,(q,1)}(\R_+\times \R^d)$ (with $2-2/q=-1+d/p$).  
%Note that, in contrast with  \cite{DW}, the velocity is in $L_1(L_\infty)$ is needed
%not only for proving  the uniqueness property, but also for controlling derivatives of the density
%since we do not know how to work with only bounded densities  in the framework of (INSK). 
\medbreak
In what follows, we will focus on the case where the density tends to $1$
at infinity and, to simplify the presentation, 
we will assume that $\mu(1)=k(1)=1.$
Note that this can be achieved by rescaling 
if $\bmu^2=\bk$ (with $\bmu:=\mu(1)$ and $\bk:= k(1)$)
as one can set $$(\tilde \rho, \tilde u, \tilde P)(t,x):=\left( \rho,\frac{u}{\bmu}, \frac{P}{\bmu^2}\right)
\left(\frac{t}{\bmu^2},\frac{x}{\bmu}\right)\cdotp$$
If this relation is not satisfied, then one cannot assume that both 
coefficients are equal to $1.$ However, proving Theorem \ref{them1nsk} 
works essentially the same.

\section{Prior estimates}\label{section2}

Now, we begin to consider the prior estimates of 
system $(INSK).$ Clearly, if $\mu$ is a constant and  $k\equiv0$ then the 
 system $(INSK)$ is just the incompressible inhomogeneous Navier-Stokes equations $(INS).$
%In the 3D case, P.~Zhang \cite{Zhang19} established the
%global existence of solutions to the 3D inhomogeneous incompressible Navier-Stokes system   with initial density in $L_\infty(\R^3),$   bounded away from zero, and  initial velocity sufficiently small in the critical Besov space $\B^{1/2}_{2,1}(\R^{3}).$ Very 
Recently, in \cite{DW}, the author work with R. Danchin completed the uniqueness part of P.~Zhang's solution in \cite{Zhang19} and proved the global existence and uniqueness of solutions to the inhomogeneous incompressible Navier-Stokes system in the case where 
\emph{the initial density $\rho_0$  is discontinuous and the initial velocity $u_0$ 
has critical  regularity.}
\smallbreak
As the works presented in \cite{DW}, the proofs rely on interpolation results, time-weighted estimates, and maximal regularity estimates in Lorentz spaces (with respect to the time variable) for the evolutionary Stokes system.

%The rest of the  paper unfolds as follows. In the next section, we introduce the energy estimates and present our main results, additionally, we also show that the main idea to deal with the system (INSK). Section \ref{section2dnsk} is devoted to the prior estimates of Theorem \ref{themnsk12d}. In section \ref{section3dnsk}, we present the estimates in the 3D case.  The proof of the existence part of Theorem \ref{themnsk12d} and Theorem \ref{themnsk13d} will be established in Section \ref{existencensk}.

However, unlike for $(INS),$ although we have the   energy balance \eqref{eq:enes} at our disposal, we do not know how to produce global weak solutions \emph{\`a la Leray}
for $(INSK),$ owing to the capillary term $\div(k(\rho)\nabla\rho\otimes\nabla\rho).$ 
The main difficulty is that   \eqref{eq:enes} 
just ensures a uniform control of  $\|\nabla\rho\|_{L_2}.$
This is not enough to pass to the limit in the capillary term from a sequence of approximate solutions that just satisfy the energy balance \eqref{eq:enes}.
%Worse than that, even for very smooth and small data,  it is not clear that using higher order energy estimates allows to construct global solutions. 

\medbreak
As discussed in \cite{TSW}, recall that in the compressible situation, to overcome this difficulty (to handle the capillary term), it is common practice to use the change of unknown (see e.g. \cite{BH}) with $\nu>0:$
\begin{equation}
    \label{eq:defvc}
    v=u+\nu\nabla \log \rho.
\end{equation}

The advantage is that after this transformation, the equation of $\rho$ becomes a heat equation, namely:
\begin{equation}
    \label{eq:rh}\d_t \rho-\nu\Delta \rho+\div(\rho v)=0.
\end{equation}
%\begin{remark}
%    Consequently, it is possible to combine maximal regularity estimates
%both for the density and the (modified) velocity so as to prove optimal well-posedness statements. 
%\end{remark}
On the other hand, from the mass equation, we have
$$\rho(\nabla \log \rho)_t+\rho u\cdot \nabla (\nabla \log \rho)+\div(\rho Du)=0,$$ which together with the momentum equation of $(INSK)$ implies
\begin{equation}
    \label{eq:nskeqvc1}
    \rho v_t+\rho v\cdot \nabla v- \div(\mu(\rho)D(v))+\nabla P=\phi
\end{equation}
where 
\begin{multline}
      \label{eq:deff}
\phi:=
-\div(k(\rho)\nabla \rho \otimes \nabla \rho)-\nu\nabla \rho \cdot Dv+\nu\nabla \rho\cdot \nabla v\\+\nu^2 \nabla \rho \cdot \nabla(\nabla \log \rho)+
2\nu\div(\mu(\rho)\nabla (\nabla \log \rho))
\end{multline}
with
\begin{multline}
\label{eq:divmr}
    \div(\mu(\rho)\nabla (\nabla \log \rho))=-\frac{\mu'(\rho)}{\rho^2}\abs{\nabla \rho}^2\nabla \rho+\frac{\mu'(\rho)}{\rho}\nabla \rho \cdot \nabla (\nabla \rho)\\
    +2\frac{\mu(\rho)}{\rho^3}\abs{\nabla \rho}^2\nabla \rho-\frac{\mu(\rho)}{\rho^2}\Delta \rho\cdot \nabla \rho-\frac{2\mu(\rho)}{\rho^2}\nabla \rho \cdot \nabla (\nabla \rho)+\frac{\mu(\rho)}{\rho^2}\rho\Delta \nabla \rho.
\end{multline}
\begin{remark}
    For the last term of \eqref{eq:divmr}, 
    when applying the maximal regularity to \eqref{eq:rh} and \eqref{eq:nskeqvc1}, we cannot bound $\frac{\mu(\rho)}{\rho}\Delta \nabla \rho.$
    Therefore, we use $\frac{\mu(\rho)}{\rho^2}\rho\Delta \nabla \rho$ to replace it.
\end{remark}
%to handle the last term, we have to choose $\mu(\rho)=\rho(\rho-1)$ or $\mu(\rho)=\ep \rho.$ In the following, we act with  $\mu(\rho)=\rho(\rho-1)$ as example.

\medbreak
As $\div v$ is not equal to zero, we may not directly use Proposition \ref{pro:mrforlame} or Proposition \ref{propregularity} to handle the pressure term. 
Therefore, to deal with the pressure term,  we rewrite the momentum equation of $(INSK)$ as a Stokes system with source term:
\begin{equation}
\label{eq:uh}
    \d_t u-\Delta u+\nabla P=-(\rho-1)\d_t u-\rho u\cdot \nabla u+g
\end{equation}
with $$g=\div((\mu(\rho)-\mu(1))D(u))-\div(k(\rho)\nabla \rho\otimes \nabla \rho).$$
%\begin{remark}
 %   As $\div u=0,$ for pressure, we have
%    $$\nabla P=$$
%\end{remark}
We then infer that, by using the transformation \eqref{eq:defvc}, the system $(INSK)$ be rewritten as:
\begin{equation}
\label{sys:ruv}
\left\{\begin{aligned}
&\rho_t+\div(\rho u)=0,\\
&\d_t\nabla \rho-\nu\Delta \nabla \rho=-\nabla\div(\rho v),
\\
 &  \d_t u-\Delta u+\nabla P=-(\rho-1)\d_t u-\rho u\cdot \nabla u+g, \\
 &  \d_t v-\Delta v-\nabla \div v=-(\rho-1)\d_t v-\rho v\cdot \nabla v+\phi -\nabla P+\div((\mu(\rho)-\mu(1))D(v)),\\
&\div u=0,\\
&(\rho,u)|_{t=0}=(\rho_{0},u_{0}),
\end{aligned}\right.
\end{equation}
where the equation of $\nabla \rho$ is a heat equation, the equation of $u$ is Stokes equation, the equation $v$ is a Lam\'e equation with $\mu=\mu'=1$. 
In the following, we obtain the prior estimates of $(INSK)$ by considering the prior estimates of \eqref{sys:ruv}.

\begin{remark}
  Without loss of generality, in the proof of the prior estimates, we choose $\nu=1.$
\end{remark}

\begin{proposition}\label{prop0d2}
Let $(\rho, u, v)$ be a smooth solution of  \eqref{sys:ruv} on $\R_+\times\R^d$
with sufficiently decaying velocity, and density satisfying
\begin{equation}\label{eq:smallrho1}\sup_{t\in\R_+}\norm{\rho(t)-1}_{L_{\infty}}\leq c.\end{equation}
  Then, for  all indices  $1<p,q,m,s<\infty$ satisfying 
 \begin{equation}
     \label{eq:defsm}
    \frac{d}{2p}+\frac{1}{q}=\frac{3}{2} \andf \frac{d}{2m}+\frac{1}{s}=\frac 12 \with   s\in(q,\infty), m\in(p,\infty),
 \end{equation}
 it holds that 
\begin{multline}\label{eq:u}
\norm{u, \nabla \rho, v/2}_{L_{\infty}(\R_+;
\B^{-1+2/p}_{p,1})}+\norm{u_{t}, (\nabla \rho)_t, (v/2)_t, \nabla^{2} u, \nabla^{2} (\nabla \rho), \nabla^{2} v/2, \nabla P/2}_{L_{q,1}(\R_+;L_{p})}\\+\norm{u, \nabla \rho, v/2}_{L_{s,1}(\R_+;L_{m})}
\leq C \norm{u_{0},\nabla \rho_0}_{\B^{-1+d/p}_{p,1}}
\end{multline}
for a positive constant $C.$  Furthermore, we have
  \begin{equation}\label{eq:dotu2} 
    \norm{\t{u}}_{L_{q,1}(\R_+;L_{p})} \leq C \norm{u_{0},\nabla \rho_0}_{\B^{-1+d/p}_{p,1}}
\end{equation}
and 
 \begin{equation}\label{eq:uLinfty}\norm{u, \nabla \rho, v }_{L_2(0,T;L_{\infty})} \leq C \norm{u_{0},\nabla \rho_0}_{\B^{-1+d/p}_{p,1}}\cdotp\end{equation}
\end{proposition}
\begin{proof} 
For system \eqref{sys:ruv}, by separately applying  Proposition 2.1 in \cite{DM2} to the equation of $\nabla \rho,$ Proposition \ref{propregularity} to the equation of $u,$   Proposition \ref{pro:mrforlame} to the equation of $v$ and multiplying the equation of $v$ with $1/2$ (to handle the pressure term in the equation of $v$),  gives us
 \begin{multline}\label{esb2d}
     \norm{u, \nabla \rho, v/2}_{L_{\infty}(\R_+;
\B^{-1+d/p}_{p,1})}+\norm{u,\nabla \rho, v/2}_{L_{s,1}(\R_+;L_{m})}\\
+\norm{u_{t}, \nabla^{2} u, v_{t}/2, \nabla^{2} v/2,(\nabla \rho)_{t}, \nabla^{2}(\nabla \rho),\nabla P/2}_{L_{q,1}(\R_+;L_{p})}\\
\lesssim \norm{u_{0}, \nabla \rho_0}_{\B^{-1+d/p}_{p,1}}+\norm{(\rho-1)u_{t}+\rho u\cdot \nabla u}_{L_{q,1}(\R_+;L_{p})}+\norm{g}_{L_{q,1}(\R_+;L_{p})}\\
+\norm{(\rho-1)v_{t}+\rho v\cdot \nabla v}_{L_{q,1}(\R_+;L_{p})}+\norm{\nabla(\div (\rho v))}_{L_{q,1}(\R_+;L_{p})}\\+\norm{\phi}_{L_{q,1}(\R_+;L_{p})}+\norm{\div((\mu(\rho)-\mu(1))D(v))}_{L_{q,1}(\R_+;L_{p})}\cdotp
 \end{multline}
By H\"older inequality, we have
$$\displaylines{\quad
\norm{(\rho-1)u_{t}+\rho u\cdot \nabla u}_{L_{q,1}(\R_+;L_{p})}+
\norm{(\rho-1)v_{t}+\rho v\cdot \nabla v}_{L_{q,1}(\R_+;L_{p})}\hfill\cr\hfill\leq
%&\,\norm{ [(\rho-1)u_{t}+\rho u\cdot \nabla u ]}_{L_{2}(0,T\times\R^{2})}\\\lesssim &\,
\norm{\rho-1}_{L_{\infty}(\R_+\times\R^{d})}(\norm{u_{t}}_{L_{q,1}(\R_+;L_{p})}+\norm{v_t}_{L_{q,1}(\R_+;L_{p})})\hfill\cr\hfill+\norm{\rho}_{L_{\infty}(\R_+\times\R^{d})}(\norm{u\cdot \nabla u}_{L_{q,1}(\R_+;L_{p})}+\norm{v\cdot \nabla v}_{L_{q,1}(\R_+;L_{p})}).\quad}
$$
If $c$ is small enough in \eqref{eq:smallrho1}, then the second line can be absorbed by the left-hand side of \eqref{esb2d}.
 For  the convection term, by H\"older inequality and embedding 
 \begin{equation}
 \label{eq:embld}
      \W^1_p(\R^d)\hookrightarrow L_{p^*}(\R^d) \with \frac{1}{p^*}=\frac 1p-\frac 1d,
 \end{equation}
  we  have
 \begin{align*}
     \norm{u\cdot \nabla u}_{L_{q,1}(\R_+;L_{p})}\leq \norm{u}_{L_{\infty}(\R_+;L_{d})}\norm{\nabla^2 u}_{L_{q,1}(\R_+;L_{p})},\\
      \norm{v\cdot \nabla v}_{L_{q,1}(0,T;L_{p})}\leq \norm{v}_{L_{\infty}(\R_+;L_{d})}\norm{\nabla^2 v}_{L_{q,1}(\R_+;L_{p})}.
 \end{align*}
For the last term in the right hand side of \eqref{esb2d}, the first-order Taylor formula gives
\begin{equation}
\label{eq:mr}
    \div((\mu(\rho)-\mu(1))D(v))=\mu'(\rho) \nabla \rho \cdot D(v)+\mu'(\theta+(1-\theta )\rho)(\rho-1) (\Delta v+\nabla \div v)
\end{equation}
with $0<\theta<1.$
Hence, note that the fact $\nabla \div=\q \Delta$ with $\q$ is the Helmholtz projector  and continuous on $L_{q,1}(L_p)$, we drives from $\mu$ is a smooth function that
\begin{multline*}
    \|\div((\mu(\rho)-\mu(1))D(v))\|_{L_{q,1}(\R_+;L_{p})}\lesssim  \|\Delta v\|_{L_{q,1}(\R_+;L_{p})}\|\rho-1\|_{L_{\infty}(\R_+\times \R^d)}\\
    +\|\nabla \rho\cdot D(v)\|_{L_{q,1}(\R_+;L_{p})}.
\end{multline*}
By a similar proof to the convection term, the estimate of the last term in the above inequality is as follows: 
$$\|\nabla \rho\cdot D(v)\|_{L_{q,1}(\R_+;L_{p})}\leq \|\nabla \rho\|_{L_{\infty}(\R_+;L_{d})}\|D(v)\|_{L_{q,1}(\R_+;L_{p^*})},$$ then use \eqref{eq:embld} again to get
\begin{multline*}
    \|\div((\mu(\rho)-\mu(1))D(v))\|_{L_{q,1}(\R_+;L_{p})}\lesssim\|\nabla \rho\|_{L_{\infty}(\R_+;L_{d})}\|\nabla^2 v\|_{L_{q,1}(\R_+;L_{p})}\\
    +\|\Delta v\|_{L_{q,1}(\R_+;L_{p})}\|\rho-1\|_{L_{\infty}(\R_+\times \R^d)}.
\end{multline*}
Similarly, note that $\div u=0,$ it is easy to check that 
\begin{multline}
\label{eq:muu}
    \|\div((\mu(\rho)-\mu(1))D(u))\|_{L_{q,1}(\R_+;L_{p})}\lesssim\|\nabla \rho\|_{L_{\infty}(\R_+;L_{d})}\|\nabla^2 u\|_{L_{q,1}(\R_+;L_{p})}\\
    +\|\Delta u\|_{L_{q,1}(\R_+;L_{p})}\|\rho-1\|_{L_{\infty}(\R_+\times \R^d)}.
\end{multline}

Now, to handle the capillary term on the right hand side of \eqref{esb2d}, we first need the decomposition
\begin{equation}
    \label{eq:krho}
     \div(k(\rho)\nabla \rho \otimes \nabla \rho)=\\
    k'(\rho)|\nabla \rho|^2 \nabla \rho+k(\rho) \Delta \rho\cdot \nabla \rho+k(\rho)\nabla \rho\cdot \nabla^2 \rho.
\end{equation}
Then, in light of \eqref{eq:embld} one get
\begin{equation}\label{eq:capes}
    \begin{aligned}
       \|\nabla \rho\cdot \nabla^2 \rho,~ \Delta \rho\cdot \nabla \rho\|_{L_{q,1}(\R_+;L_{p})}&\leq \|\nabla \rho\|_{L_{\infty}(\R_+;L_{d})}\|\nabla^2 \rho\|_{L_{q,1}(\R_+;L_{p*})}\\&\lesssim \|\nabla \rho\|_{L_{\infty}(\R_+;L_{d})}\|\nabla^2(\nabla \rho)\|_{L_{q,1}(\R_+;L_{p})} ;\\
      \||\nabla \rho|^2 \nabla \rho\|_{L_{q,1}(\R_+;L_{p})}&\leq \|\nabla \rho\|_{L_{\infty}(\R_+;L_{d})}\||\nabla \rho|^2\|_{L_{q,1}(\R_+;L_{p^*}}) \\
      &\lesssim \|\nabla \rho\|_{L_{\infty}(\R_+;L_{d})}\|\nabla \rho\cdot \nabla^2\rho\|_{L_{q,1}(\R_+;L_{p})}\\
      &\lesssim \|\nabla \rho\|^2_{L_{\infty}(\R_+;L_{d})}\|\nabla^2(\nabla \rho)\|_{L_{q,1}(\R_+;L_{p})}.
    \end{aligned}
\end{equation}

Hence, $k$ is a smooth function leads to
\begin{equation}
\label{eq:kr}
\begin{aligned}
    &\|\div(k(\rho)\nabla \rho \otimes \nabla \rho)\|_{L_{q,1}(\R_+;L_{p})}\\ \lesssim& \|\nabla^2(\nabla \rho)\|_{L_{q,1}(\R_+;L_{p}(\R^d))}\|\nabla \rho\|_{L_{\infty}(\R_+;L_{d})}(1+\|\nabla \rho\|_{L_{\infty}(\R_+;L_{d})}).
\end{aligned}
\end{equation}
Furthermore, we can conclude from \eqref{eq:muu} and \eqref{eq:kr} that
\begin{multline*}
    \|g\|_{L_{q,1}(\R_+;L_{p})}\lesssim \|\nabla \rho\|_{L_{\infty}(\R_+;L_{d})}\|\nabla^2 u\|_{L_{q,1}(\R_+;L_{p})}
    +\|\Delta u\|_{L_{q,1}(\R_+;L_{p})}\|\rho-1\|_{L_{\infty}(\R_+\times \R^d)}\\
    \|\nabla^2(\nabla \rho)\|_{L_{q,1}(\R_+;L_{p}(\R^d))}\|\nabla \rho\|_{L_{\infty}(\R_+;L_{d})}(1+\|\nabla \rho\|_{L_{\infty}(\R_+;L_{d})}).
\end{multline*}
Let us turn to the estimates of $\phi,$ at first, we then have by \eqref{eq:embld}
\begin{equation}
\label{eq:esrestf}
    \begin{aligned}
     \|\nabla \rho\cdot Dv\|_{L_{q,1}(\R_+;L_{p})}\lesssim \|\nabla \rho\|_{L_{\infty}(\R_+;L_{d})}\|\nabla^2 v\|_{L_{q,1}(\R_+;L_{p})};\\
     \|\nabla \rho\cdot \nabla v\|_{L_{q,1}(\R_+;L_{p})}\lesssim \|\nabla \rho\|_{L_{\infty}(\R_+;L_{d})}\|\nabla^2 v\|_{L_{q,1}(\R_+;L_{p})}.
\end{aligned}
\end{equation}
In what follows, to handle the last term in \eqref{eq:divmr} , the following inequality  from embedding  $\B^{d/p}_{p,1}(\R^d)\hookrightarrow L_\infty (\R^d)$ and  Proposition \ref{dervieq} plays a key role:
\begin{equation}
    \label{eq:esimpot}
    \|\rho\|_{L_\infty(\R_+\times \R^d)}\lesssim  \|\rho\|_{L_\infty(\R_+;  \B^{d/p}_{p,1})}\lesssim \|\nabla \rho\|_{L_\infty(\R_+; \B^{-1+d/p}_{p,1})},
\end{equation}
we emphasize that this inequality will also be frequently used to deal with the source term of the heat equation of $\nabla \rho.$
Note that the decomposition 
\begin{equation}
\label{eq:nnr}
\nabla(\nabla \log \rho)=-\frac{1}{\rho^2} \abs{\nabla \rho}^2+\frac{1}{\rho} \nabla^2 \rho,
\end{equation}
and \eqref{eq:divmr}, by taking advantage of \eqref{eq:capes},  \eqref{eq:esimpot},  and $\mu$ is a smooth function, we hence deduce
\begin{multline}
    \label{eq:esrest}
    \|\nabla(\nabla \log \rho)\|_{L_{q,1}(\R_+;L_{p})}+\|\div(\mu(\rho) \nabla(\nabla \log \rho))\|_{L_{q,1}(\R_+;L_{p})}
    \\
    \lesssim \|\nabla^2(\nabla \rho)\|_{L_{q,1}(\R_+;L_{p})}\|\nabla \rho\|_{L_{\infty}(\R_+;L_{d})}(1+\|\nabla \rho\|_{L_{\infty}(\R_+;L_{d})})\\+\|\nabla \rho\|_{L_\infty(\R_+; \B^{-1+d/p}_{p,1})}\|\nabla^2(\nabla \rho)\|_{L_{q,1}(\R_+;L_{p})}
\end{multline}
Lastly, we finish the estimates of $\phi$ by \eqref{eq:kr}, \eqref{eq:esrestf} and \eqref{eq:esrest},
\begin{multline*}
    \|\phi\|_{L_{q,1}(\R_+;L_{p})}\lesssim \|\nabla^2(\nabla \rho)\|_{L_{q,1}(\R_+;L_{p})}\|\nabla \rho\|_{L_{\infty}(\R_+;L_{d})}(1+\|\nabla \rho\|_{L_{\infty}(\R_+;L_{d})})\\+\|\nabla \rho\|_{L_{\infty}(\R_+;L_{d})}\|\nabla^2 v\|_{L_{q,1}(\R_+;L_{p})}+\|\nabla \rho\|_{L_\infty(\R_+; \B^{-1+d/p}_{p,1})}\|\nabla^2(\nabla \rho)\|_{L_{q,1}(\R_+;L_{p})}.
\end{multline*}

Now, we start to bound the source term of the heat equation of $\nabla \rho.$ It is clear that we have
\begin{equation}
    \label{eq:ndrv}
    \nabla \div(\rho v)=v\cdot \nabla (\nabla \rho)+\nabla \rho\cdot \nabla v+\nabla \rho\cdot \div v+\rho \nabla \div v.
\end{equation}
Thanks to the embeddings \eqref{eq:embld} and \eqref{eq:esimpot}, one get
\begin{align*}
 \|\nabla \rho\cdot \div v\|_{L_{q,1}(\R_+;L_{p})}&\lesssim \|\nabla \rho\|_{L_{\infty}(\R_+;L_{d})}\|\nabla^2 v\|_{L_{q,1}(\R_+;L_{p})}\\
    \norm{v\cdot \nabla (\nabla \rho)}_{L_{q,1}(\R_+;L_{p})}&\lesssim \norm{v}_{L_{\infty}(\R_+;L_{d})}\norm{\nabla^2 (\nabla \rho)}_{L_{q,1}(\R_+;L_{p})},\\ 
    \norm{\rho \nabla \div v}_{L_{q,1}(\R_+;L_{p})}&\lesssim  \norm{\rho}_{L_{\infty}(\R_+\times \R^d)}\norm{\nabla^2 v}_{L_{q,1}(\R_+;L_{p})},\\
    &\lesssim  \norm{\nabla \rho}_{L_{\infty}(\R_+;\B^{-1+d/p}_{p,1})}\norm{\nabla^2 v}_{L_{q,1}(\R_+;L_{p})}
\end{align*}
from which together with \eqref{eq:esrestf} gives
\begin{multline*}
    \norm{\nabla(\div (\rho v))}_{L_{q,1}(\R_+;L_{p})}\lesssim \|\nabla \rho\|_{L_{\infty}(\R_+;L_{d})}\|\nabla^2 v\|_{L_{q,1}(\R_+;L_{p})}\\
    +\norm{v}_{L_{\infty}(\R_+;L_{d})}\norm{\nabla^2 (\nabla \rho)}_{L_{q,1}(\R_+;L_{p})}+\norm{\nabla \rho}_{L_{\infty}(\R_+;\B^{-1+d/p}_{p,1})}\norm{\nabla^2 v}_{L_{q,1}(\R_+;L_{p})}
\end{multline*}

Putting all the estimates together and  recalling $\rho$ such that \eqref{eq:smallrho1}, then all the terms with $\|\rho-1\|_{L_\infty(\R_+\times \R^d)}$ can be absorbed by the left hand of \eqref{esb2d} and we see that
 \begin{multline*}
     \norm{u, \nabla \rho, v/2}_{L_{\infty}(\R_+;
\B^{-1+2/p}_{p,1})}+\norm{u,\nabla \rho, v/2}_{L_{s,1}(\R_+;L_{m})}\\
+\norm{u_{t}, \nabla^{2} u, v_{t}/2, \nabla^{2} v/2,(\nabla \rho)_{t}, \nabla^{2}(\nabla \rho),\nabla P/2}_{L_{q,1}(\R_+;L_{p})}\\
\leq C\bigl(\norm{u_{0}, \nabla \rho_0}_{\B^{-1+2/p}_{p,1}}+\norm{u}_{L_{\infty}(\R_+;L_{d})}\norm{\nabla^2 u}_{L_{q,1}(\R_+;L_{p})}+\norm{v}_{L_{\infty}(\R_+;L_{d})}\norm{\nabla^2 v}_{L_{q,1}(\R_+;L_{p})}\\
+ \|\nabla \rho\|_{L_{\infty}(\R_+;L_{d})}(\|\nabla^2 u\|_{L_{q,1}(\R_+;L_{p})}+\|\nabla^2 v\|_{L_{q,1}(\R_+;L_{p})})\\
+\|\nabla \rho\|_{L_\infty(\R_+; \B^{-1+d/p}_{p,1})}\|\nabla^2(\nabla \rho)\|_{L_{q,1}(\R_+;L_{p})}\\
+
\|\nabla^2(\nabla \rho)\|_{L_{q,1}(\R_+;L_{p})}\|\nabla \rho\|_{L_{\infty}(\R_+;L_{d})}(1+\|\nabla \rho\|_{L_{\infty}(\R_+;L_{d})})\\
+\norm{v}_{L_{\infty}(\R_+;L_{d})}\norm{\nabla^2 (\nabla \rho)}_{L_{q,1}(\R_+;L_{p})}+\norm{\nabla \rho}_{L_{\infty}(\R_+;\B^{-1+d/p}_{p,1})}\norm{\nabla^2 v}_{L_{q,1}(\R_+;L_{p})}\bigr)\cdotp
 \end{multline*}
Denote 
\begin{multline*}
    X=:\underset{t\in \R_+}{\sup} \norm{u(t), \nabla \rho(t), v(t)/2}_{
\B^{-1+2/p}_{p,1}}+\norm{u,\nabla \rho, v/2}_{L_{s,1}(\R_+;L_{m})}\\
+\norm{u_{t}, \nabla^{2} u, v_{t}/2, \nabla^{2} v/2,(\nabla \rho)_{t}, \nabla^{2}(\nabla \rho),\nabla P/2}_{L_{q,1}(\R_+;L_{p})}.
\end{multline*}
Finally, in the 2D case, by virtue of \eqref{eq:enes} and embedding
\begin{equation}
    \label{eq:emdld}
    \B^{-1+d/p}_{p,1}(\R^d)\hookrightarrow L_d(\R^d),
\end{equation}
we can write that
$$X\leq C(X_0+(1+E_0)X^2).$$
Thus, by a bootstrap argument, one can deduce that if
\begin{equation}
\label{eq:ini2d}
    4CX_0(1+E_0)\leq 1,
\end{equation}
then we get
\begin{equation}
     \label{eq:esxd2}
    X\leq 2X_0.
\end{equation}
   
Furthermore, in the three-dimensional case, the embedding \eqref{eq:emdld} guarantees that
$$X\leq C(X_0+X^2+X^3).$$
Using again the  standard bootstrap argument and assuming that if (with a different constant $C$)
\begin{equation}
    \label{eq:ini3d}
    CX_0\leq1,
\end{equation}
then one has
 \begin{equation}
    \label{eq:esxd3}
    X\lesssim X_0.
\end{equation}

In order to prove \eqref{eq:dotu2}, it suffices to use the fact that
$$\begin{aligned}
  \norm{\t{u}}_{L_{q,1}(\R_+;L_{p})} &\leq  \norm{u_{t}}_{L_{q,1}(\R_+;L_{p})}\!+\!\norm{u\cdot \nabla u}_{L_{q,1}(\R_+;L_{p})}\\
  \leq& \norm{u_{t}}_{L_{q,1}(\R_+;L_{p})}\!+\!\norm{u}_{L_{\infty}(\R_+;L_{d})}\norm{\nabla^2 u}_{L_{q,1}(\R_+;L_{p})}.
  \end{aligned}$$
  Bounding the right-hand side according to \eqref{eq:emdld}, \eqref{eq:esxd2} and  \eqref{eq:esxd3}
  yields \eqref{eq:dotu2}. 
  \smallbreak
  Finally, as a consequence of Gagliardo-Nirenberg inequality and embedding \eqref{eq:emdld},  we have:
  \begin{equation*}
  \norm{z}_{L_{\infty}}\lesssim \norm{z}^{1-q/2}_{L_{d}}\norm{\nabla ^{2}z}^{q/2}_{L_{p}}\lesssim \norm{z}^{1-q/2}_{\B^{-1+d/p}_{p,1}}\norm{\nabla ^{2}z}^{q/2}_{L_{p}}.
  \end{equation*}
Hence,   using Inequality \eqref{eq:esxd2} and  \eqref{eq:esxd3},  we find that 
\begin{equation*}
    \begin{aligned}
    \int_{0}^{\infty}\norm{u}^{2}_{L_{\infty}}\,dt &\leq C\int_{0}^{\infty}\norm{u}^{2-q}_{\B^{-1+d/p}_{p,1}}\norm{\nabla^{2}u}^{q}_{L_{p}}\,dt\\
    &\leq C\norm{u}^{2-q}_{L_{\infty}(\R_+;\B^{-1+d/p}_{p,1})}\norm{\nabla^{2}u}^{q}_{L_{q}(\R_+;L_{p})} \\
    &\leq CX_0^2.    \end{aligned}\end{equation*}

Note that by arguing similarly, 
one  can easily get  
$$\|\nabla \rho\|_{L_2(\R_+;L_\infty)}, \|v\|_{L_2(\R_+;L_\infty)}\leq C X_0,$$
this completes the proof of the proposition.\end{proof}

\begin{remark}
As the vector field $u$ is divergence-free, the density satisfies \eqref{eq:rlin}.

\end{remark}
\begin{remark}
    We find that the method to obtain the global existence and uniqueness of $(INS)$ with large initial data (2D case) in \cite{DW} is not suitable for system $(INSK).$ On the one hand, because of the capillary term:
    \begin{align*}
    \|\nabla \rho\cdot \nabla^2 \rho\|_{L_{q,1}(\R_+;L_{p})}&\leq \|\nabla \rho\|_{L_{q_1,1}(\R_+;L_{p_1})}\|\nabla^2 \rho\|_{L_{2}(\R_+\times\R^2)} ;\\
     \|\Delta \rho\cdot \nabla \rho\|_{L_{q,1}(\R_+;L_{p})}&\leq \|\nabla \rho\|_{L_{q_1,1}(\R_+;L_{p_1})}\|\nabla^2 \rho\|_{L_{2}(\R_+\times\R^2)} ;\\
      \||\nabla \rho|^2 \nabla \rho\|_{L_{q,1}(\R_+;L_{p})}&\leq \|\nabla \rho\|_{L_{q_1,1}(\R_+;L_{p_1})}\|\nabla \rho\|^2_{L_{4}(\R_+\times\R^2)}\\&\leq \|\nabla \rho\|_{L_{q_1,1}(\R_+;L_{p_1})}\|\nabla \rho\|_{L_{\infty}(\R_+;L_2)}\|\nabla^2 \rho\|_{L_{2}(\R_+\times\R^2)}.
\end{align*}
Even though we get the estimates of  $\|\nabla^2 \rho\|_{L_{2}(\R_+\times\R^2)}$ by Proposition \ref{prop:r2}, the smallness condition is still required (possibly, move the smallness condition of $\|u_0\|_{\B^{-1+2/p}_{p,1}(\R^2)},$ see Proposition \ref{prop0ad2} ). On the other hand, the way to deal with the last term in \eqref{eq:divmr} and the source term of the heat equation of $\nabla \rho$ in \eqref{sys:ruv} are also show that the smallness condition is necessary.
\end{remark}

\begin{proposition}\label{prop1}
Under the  assumptions of Proposition  \ref{prop0d2}, we have
\begin{multline}
\label{eq:twesu}
    \norm{tu, tv/2, t(\nabla \rho)}_{L_{\infty}(\R_+;\B^{2-2/s}_{m,1})}+\norm{(tv/2)_{t},  \nabla^{2}(tv/2),  \nabla (tP)/2}_{L_{s,1}(\R_+;L_{m})}
\\
+\norm{(tu)_{t}, t\dot u, \nabla^{2}(tu),(t\nabla \rho)_{t}, \nabla^{2}(t\nabla \rho)}_{L_{s,1}(\R_+;L_{m})}
\leq  C \norm{u_{0},\nabla \rho_0}_{\B^{-1+d/p}_{p,1}}.
\end{multline}
\end{proposition}
\begin{proof}
Now, multiplying both sides of \eqref{sys:ruv} by time $t$ except the mass equation yields
\begin{equation*}
\left\{\begin{aligned}
&\d_t(t\nabla \rho)-\nu\Delta t\nabla \rho=-t\nabla\div(\rho v)+\nabla \rho,
\\
 &  \d_t (tu)-\Delta tu+t\nabla P=-(\rho-1)\d_t (tu)+\rho u-\rho u\cdot \nabla u+g, \\
 &  \d_t (tv)-\Delta( tv)-\nabla\div(tv)=-(\rho-1)\d_t (tv)-t\rho v\cdot \nabla v+t\phi+\rho v \\
 &\hspace{3cm}-t\nabla P+t\div((\mu(\rho)-\mu(1))D(v))\\
  &  \div(tu)=0.
\end{aligned}\right.
\end{equation*}
Similar to \eqref{sys:ruv}, we apply Proposition 2.1 in  \cite{DM2} to the equation of $t\nabla \rho,$ Proposition \ref{propregularity} to the equation of $tu,$ Proposition \ref{pro:mrforlame} to the equation of $tv$ and multiplying the equation of $tv$ with $1/2$ respectively, to get
 \begin{multline}\label{esb2dtw}
     \norm{tu, tv/2, t(\nabla \rho)}_{L_{\infty}(\R_+;\B^{2-2/s}_{m,1})}+\norm{(tu)_{t}, \nabla^{2}(tu), \nabla (tP)/2}_{L_{s,1}(\R_+;L_{m})}\\
     +\norm{(tv/2)_{t}, \nabla^{2}(tv/2), (t\nabla \rho)_{t}, \nabla^{2}(t\nabla \rho)}_{L_{s,1}(\R_+;L_{m})}
\lesssim \norm{(\rho-1)(tu)_{t}+t\rho u\cdot \nabla u}_{L_{s,1}(\R_+;L_{m})}\\ +\norm{(\rho-1)(tv)_{t}+t\rho v\cdot \nabla v}_{L_{s,1}(\R_+;L_{m})}
+\norm{t\phi}_{L_{s,1}(\R_+;L_{m})}+\norm{tg}_{L_{q,1}(\R_+;L_{p})}\\+\norm{t\nabla(\div (\rho v))}_{L_{s,1}(\R_+;L_{m})}+\norm{t\div((\mu(\rho)-\mu(1))D(v))}_{L_{q,1}(\R_+;L_{p})}\\+\|\nabla \rho\|_{L_{s,1}(\R_+;L_{m})}
+\|\rho u\|_{L_{s,1}(\R_+;L_{m})}+\|\rho v\|_{L_{s,1}(\R_+;L_{m})}\cdotp
 \end{multline}
We first bound the last three terms on the right-hand side by \eqref{eq:u}. Then, as $1+d/m=2-2/s,$ the H\"older inequality and
the following embedding:
\begin{equation}\label{eq:embed}
\dot B^{d/m}_{m,1}(\R^d)\hookrightarrow L_\infty(\R^d)
\end{equation} yields
\begin{align*}
    &\norm{t\rho u\cdot \nabla u}_{L_{s,1}(\R_+;L_{m})}+\norm{t\rho v\cdot \nabla v}_{L_{s,1}(\R_+;L_{m})}\\
    \lesssim & \norm{t \nabla u}_{L_{\infty}(\R_+\times \R^{d})} \norm{u}_{L_{s,1}(\R_+;L_{m})}+\norm{t \nabla v}_{L_{\infty}(\R_+\times \R^{d})} \norm{v}_{L_{s,1}(\R_+;L_{m})}\\
    \lesssim & \norm{t u}_{L_{\infty}(\R_+;\B^{2-2/s}_{m,1})} \norm{u}_{L_{s,1}(\R_+;L_{m})}+\norm{t v}_{L_{\infty}(\R_+;\B^{2-2/s}_{m,1})} \norm{v}_{L_{s,1}(\R_+;L_{m})}\cdotp
\end{align*}
Hence,
\begin{equation*}
    \begin{aligned}
    &\norm{(\rho-1)(tu)_{t}+t\rho u\cdot \nabla u}_{L_{s,1}(\R_+;L_{m})}+\norm{(\rho-1)(tv)_{t}+t\rho v\cdot \nabla v}_{L_{s,1}(\R_+;L_{m})}\\
    \lesssim &\|\rho-1\|_{L_\infty(\R_+\times \R^d)}(\norm{(tu)_{t}}_{L_{s,1}(\R_+;L_{m})}+\norm{(tv)_{t}}_{L_{s,1}(\R_+;L_{m})})\\
    &+
    \norm{t u}_{L_{\infty}(\R_+;\B^{2-2/s}_{m,1})} \norm{u}_{L_{s,1}(\R_+;L_{m})}+\norm{t v}_{L_{\infty}(\R_+;\B^{2-2/s}_{m,1})} \norm{v}_{L_{s,1}(\R_+;L_{m})}\cdotp
    \end{aligned}
\end{equation*}

Thanks to \eqref{eq:embld}, \eqref{eq:esimpot}, \eqref{eq:embed},  and the following interpolation inequality
$$\|\nabla z\|^2_{L_\infty(\R_+\times\R^d)}\leq \| z\|_{L_\infty(\R_+\times\R^d)}\|\nabla^2 z\|_{L_\infty(\R_+\times\R^d)},$$
we can get
\begin{equation}
    \label{eq:esusf}
    \begin{aligned}
    \|t\nabla \rho\cdot D(v)\|_{L_{s,1}(\R_+;L_{m})}&\leq \|\nabla \rho\|_{L_{s,1}(\R_+;L_{m})} \|tD(v)\|_{L_\infty(\R_+\times\R^d)}\\
    &\lesssim \|\nabla \rho\|_{L_{s,1}(\R_+;L_{m})} \|tv\|_{L_{\infty}(\R_+;\B^{2-2/s}_{m,1})},\\
      \|t\nabla \rho\cdot Dv, \ t\nabla \rho\cdot \nabla v\|_{L_{s,1}(\R_+;L_{m})}
      &\leq \|\nabla \rho\|_{L_{s,1}(\R_+;L_{m})} \|t\nabla v\|_{L_\infty(\R_+\times\R^d)}\\
    &\lesssim \|\nabla \rho\|_{L_{s,1}(\R_+;L_{m})} \|tv\|_{L_{\infty}(\R_+;\B^{2-2/s}_{m,1})},\\
    \|t\nabla \rho\cdot \nabla^2 \rho, \  t\nabla \rho\cdot \Delta \rho\|_{L_{s,1}(\R_+;L_{m})}
    &\leq \|\nabla \rho\|_{L_{s,1}(\R_+;L_{m})} \|t\nabla^2 \rho\|_{L_\infty(\R_+\times\R^d)}\\
    &\lesssim \|\nabla \rho\|_{L_{s,1}(\R_+;L_{m})} \|t\nabla \rho\|_{L_{\infty}(\R_+;\B^{2-2/s}_{m,1})},\\
    \|t|\nabla \rho|^2 \nabla \rho\|_{L_{s,1}(\R_+;L_{m})}&\leq \|\nabla \rho\|_{L_{s,1}(\R_+;L_{m})} t\||\nabla \rho|^2 \|_{L_\infty(\R_+\times\R^d)}\\
    &\lesssim \|\nabla \rho\|_{L_{s,1}(\R_+;L_{m})} \| \rho \|_{L_\infty(\R_+\times\R^d)}\|t\nabla^2 \rho \|_{L_\infty(\R_+\times\R^d)}\\
    &\lesssim \|\nabla \rho\|_{L_{s,1}(\R_+;L_{m})} \|t\nabla \rho \|_{L_{\infty}(\R_+;\B^{2-2/s}_{m,1})}\|\nabla \rho \|_{L_{\infty}(\R_+;\B^{-1+d/p}_{p,1})}.
\end{aligned}
\end{equation}
From which it follows, using \eqref{eq:mr} and \eqref{eq:krho} and the fact $\nabla \div=\q \Delta$ that
\begin{equation}
    \label{eq:estf}
    \begin{aligned}
      \|t\div((\mu(\rho)-\mu(1))D(v))\|_{L_{s,1}(\R_+;L_{m})}\lesssim& \|\nabla \rho\|_{L_{s,1}(\R_+;L_{m})} \|tv\|_{L_{\infty}(\R_+;\B^{2-2/s}_{m,1})}\\
    &+\|t\nabla^2 v\|_{L_{s,1}(\R_+;L_{m})}\|\rho-1\|_{L_{\infty}(\R_+\times \R^d)},\\
    \|t\div((\mu(\rho)-\mu(1))D(u))\|_{L_{s,1}(\R_+;L_{m})}\lesssim &
    \|\nabla \rho\|_{L_{s,1}(\R_+;L_{m})} \|tu\|_{L_{\infty}(\R_+;\B^{2-2/s}_{m,1})}\\
    &+\|t\nabla^2 u\|_{L_{s,1}(\R_+;L_{m})}\|\rho-1\|_{L_{\infty}(\R_+\times \R^d)},\\
    \|t\div(k(\rho)\nabla \rho \otimes \nabla \rho)\|_{L_{s,1}(\R_+;L_{m})} \lesssim 
    &\|\nabla \rho\|_{L_{s,1}(\R_+;L_{m})} \|t\nabla \rho \|_{L_{\infty}(\R_+;\B^{2-2/s}_{m,1})}\\
    &\cdot(1+\|\nabla \rho \|_{L_{\infty}(\R_+;\B^{-1+d/p}_{p,1})}).
\end{aligned}
\end{equation}

Hence, it turns out that
\begin{multline*}
     \|tg\|_{L_{s,1}(\R_+;L_{m})}\lesssim  \|\nabla \rho\|_{L_{s,1}(\R_+;L_{m})} \|tu\|_{L_{\infty}(\R_+;\B^{2-2/s}_{m,1})}+\|t\nabla^2 u\|_{L_{s,1}(\R_+;L_{m})}\|\rho-1\|_{L_{\infty}(\R_+\times \R^d)}\\
     +\|\nabla \rho\|_{L_{s,1}(\R_+;L_{m})} \|t\nabla \rho \|_{L_{\infty}(\R_+;\B^{2-2/s}_{m,1})}(1+\|\nabla \rho \|_{L_{\infty}(\R_+;\B^{-1+d/p}_{p,1})}).
\end{multline*}
Combing \eqref{eq:esimpot}, \eqref{eq:esusf}, \eqref{eq:estf} with the definition of $\phi$ in \eqref{eq:deff},  we can obtain
\begin{equation*}
    \begin{aligned}
    \|t\phi\|_{L_{s,1}(\R_+;L_{m})}\lesssim& \|\nabla \rho\|_{L_{s,1}(\R_+;L_{m})} \|t\nabla \rho \|_{L_{\infty}(\R_+;\B^{2-2/s}_{m,1})}(1+\|\nabla \rho \|_{L_{\infty}(\R_+;\B^{-1+d/p}_{p,1})})\\&+\|\nabla \rho\|_{L_{s,1}(\R_+;L_{m})} \|tv\|_{L_{\infty}(\R_+;\B^{2-2/s}_{m,1})}\\&+\|\rho\|_{L_\infty(\R_+\times \R^d)}\|t\nabla^2(\nabla \rho)\|_{L_{q,1}(\R_+;L_{p})}\\
    \lesssim& \|\nabla \rho\|_{L_{s,1}(\R_+;L_{m})} \|t\nabla \rho \|_{L_{\infty}(\R_+;\B^{2-2/s}_{m,1})}(1+\|\nabla \rho \|_{L_{\infty}(\R_+;\B^{-1+d/p}_{p,1})})\\&+\|\nabla \rho\|_{L_{s,1}(\R_+;L_{m})} \|tv\|_{L_{\infty}(\R_+;\B^{2-2/s}_{m,1})}\\&+\norm{\nabla \rho}_{L_{\infty}(\R_+;\B^{-1+d/p}_{p,1})}\|t\nabla^2(\nabla \rho)\|_{L_{q,1}(\R_+;L_{p})}.
\end{aligned}
\end{equation*}

Note that, to achieve \eqref{eq:twesu}, it is suffices to add the estimates of $ \norm{t\nabla(\div (\rho v))}_{L_{s,1}(\R_+;L_{m})}.$  One can conclude that the following estimates by using \eqref{eq:esimpot} and \eqref{eq:embed}:
\begin{align*}
 \|t\nabla \rho\cdot \div v\|_{L_{s,1}(\R_+;L_{m})}&\lesssim \|\nabla \rho\|_{L_{s,1}(\R_+;L_{m})}\|t\div v\|_{L_{\infty}(\R_+\times \R^d)}\\
 &\lesssim \|\nabla \rho\|_{L_{s,1}(\R_+;L_{m})}\|t v\|_{L_{\infty}(\R_+;\B^{2-2/s}_{m,1})},\\
    \norm{tv\cdot \nabla (\nabla \rho)}_{L_{s,1}(\R_+;L_{m})}&\lesssim \norm{v}_{L_{s,1}(\R_+;L_{m})}\norm{t\nabla^2 \rho}_{L_{\infty}(\R_+\times\R^d)},\\ 
    \norm{t\rho \nabla \div v}_{L_{s,1}(\R_+;L_{m})}&\lesssim  \norm{\rho}_{L_{\infty}(\R_+\times \R^d)}\norm{t\nabla^2 v}_{L_{s,1}(\R_+;L_{m})}\\
    &\lesssim  \norm{\nabla \rho}_{L_{\infty}(\R_+;\B^{-1+d/p}_{p,1})}\norm{t\nabla^2 v}_{L_{s,1}(\R_+;L_{m})},
\end{align*}
which along with  \eqref{eq:ndrv} and \eqref{eq:esusf} infer that
\begin{multline*}
    \norm{t\nabla(\div (\rho v))}_{L_{s,1}(\R_+;L_{m})}\lesssim \norm{v}_{L_{s,1}(\R_+;L_{m})}\norm{t\nabla^2 \rho}_{L_{\infty}(\R_+\times\R^d)}\\
    +\norm{\nabla \rho}_{L_{\infty}(\R_+;\B^{-1+d/p}_{p,1})}\norm{t\nabla^2 v}_{L_{s,1}(\R_+;L_{m})}
    +\|\nabla \rho\|_{L_{s,1}(\R_+;L_{m})} \|tv\|_{L_{\infty}(\R_+;\B^{2-2/s}_{m,1})}.
\end{multline*}

 In the end, we use \eqref{eq:smallrho1} to absorb all the term with $\|\rho-1\|_{L_\infty(\R_+\times \R^d)}$ by the left hand side of \eqref{esb2dtw}, while it follows from  \eqref{eq:u}, \eqref{eq:ini2d}, \eqref{eq:ini3d} and Proposition \ref{prop0d2} that
\begin{multline}\label{eq:uu}
 \norm{tu, tv/2, t(\nabla \rho)}_{L_{\infty}(\R_+;\B^{2-2/s}_{m,1})}+\norm{(tv/2)_{t},  \nabla^{2}(tv/2),  \nabla (tP)/2}_{L_{s,1}(\R_+;L_{m})}
\\
+\norm{(tu)_{t}, \nabla^{2}(tu),(t\nabla \rho)_{t}, \nabla^{2}(t\nabla \rho)}_{L_{s,1}(\R_+;L_{m})}
\leq  C \norm{u_{0},\nabla \rho_0}_{\B^{-1+d/p}_{p,1}}\cdotp\end{multline}
To bound $t\dot u,$ we just 
have to observe that $t\dot u= (tu)_t - u + tu\cdot\nabla u.$ Hence,  H\"older inequality and \eqref{eq:embed} yields
\begin{multline*}
    \norm{t\t{u}}_{L_{s,1}(\R_+;L_{m})}\leq  \norm{(tu)_{t}}_{L_{s,1}(\R_+;L_{m})}+\norm{u}_{L_{s,1}(\R_+;L_{m})}\\
+\norm{tu}_{L_\infty(\R_+; \dot B^{2-2/s}_{m,1})} \norm{u}_{L_{s,1}(\R_+;L_{m})}.
\end{multline*}
Thus, we obtain the desired result owing to inequalities \eqref{eq:u} and \eqref{eq:uu}.
\end{proof}
The results that we have in hand will enable us to bound $\nabla u,$ $\nabla^2 \rho$ are in $L_1(\R_+;L_\infty).$ %Let us first present the 2D case.
\begin{corollary}\label{coro1d2}
under the same assumption of Proposition  \ref{prop1}, we have:
\begin{align}\label{eq:tnablau1}
\int_{0}^{\infty} \norm{\nabla u}_{L_{\infty}(\R^d)}\,dt&\leq  C \norm{u_{0},\nabla \rho_0}_{\B^{-1+d/p}_{p,1}(\R^d)}\\
\label{eq:tnablau2}
 \biggl(\int_{0}^\infty t\norm{\nabla u}^{2}_{L_{\infty}(\R^d)}\,dt\biggr)^{1/2}&\leq  C \norm{u_{0},\nabla \rho_0}_{\B^{-1+d/p}_{p,1}(\R^d)} \\
\label{estulilid2}\underset{t\in \R_+}
\sup ( t)^{1/2} \|u(t)\|_{L_\infty(\R^d)}&\leq C
 \norm{u_{0},\nabla \rho_0}_{\B^{-1+d/p}_{p,1}(\R^d)}.\end{align}
\end{corollary}
\begin{proof} 
In the 2D case, according to the following Gagliardo-Nirenberg inequality
$$\norm{z}_{L_{\infty}(\R^2)}\lesssim \norm{ \nabla z}^{1-2/m}_{L_{p}(\R^2)}\norm{\nabla z}^{2/m}_{L_{m}(\R^2)},$$
and H\"older estimates in Lorentz spaces
(see Proposition \ref{p:lorentz}), we arrive at
\begin{equation*}
\begin{aligned}
\int^{\infty}_{0}\norm{\nabla u}_{L_{\infty}(\R^2)}\,dt&\lesssim \int^{\infty}_{0} t^{-2/m} \norm{ \nabla^{2} u}^{1-2/m}_{L_{p}(\R^2)}\norm{t\nabla^{2} u}^{2/m}_{L_{m}(\R^2)}\,dt\\
&\lesssim\norm{t^{-2/m}}_{L_{m/2,\infty}(\R_+)}\norm{\nabla^{2} u}^{1-2/m}_{L_{q,1}(\R_+;L_{p}(\R^2))}\norm{t \nabla^{2} u}^{2/m}_{L_{s,1}(\R_+;L_{m}(\R^2))}.\end{aligned}
\end{equation*}
Then, by using  $t\mapsto t^{-2/m}\in L_{m/2,\infty}(\R_+)$
and  the other terms of
the  right-hand side may be bounded by means of Propositions
\ref{prop0d2} and \ref{prop1}, we get \eqref{eq:tnablau1} in 2D case.

In the three-dimensional case, the proofs are different. As $(s,m)$ defined in \eqref{eq:defsm} implies that $p<3<m,$ we bound  $\nabla u$ take advantage of the following interpolation inequality:
$$\norm{u}_{L_{\infty}(\R^{3})}\leq \norm{\nabla u}^{\frac{p(m-3)}{3(m-p)}}_{L_{p}(\R^{3})}\norm{\nabla u}^{\frac{m(3-p)}{3(m-p)}}_{L_{m}(\R^{3})}.$$
Hence, applying the following exponents to the  H\"older inequality in Lorentz spaces:
\begin{small}$$(p_1,r_1)\!=\!\biggl(\frac{3(m\!-\!p)}{m(3-p)},\infty\!\biggr),\!\!\quad
(p_2,r_2)\!=\!\biggl(\frac{3q(m\!-\!p)}{p(m-3)},\frac{p_2}q\!\biggr),\!\!\quad
(p_3,r_3)\!=\!\biggl(\frac{3s(m\!-\!p)}{m(3-p)},\frac{p_3}s\!\biggr)$$\end{small}
and  thanks to the fact that   $t^{-\alpha}$ with $\alpha=m(3-p)/(3(m-p))$  
 is in $L_{1/\alpha,\infty}(\R_+),$ 
 \eqref{eq:u} and \eqref{eq:twesu},
one writes
\begin{equation*}
    \begin{aligned}
    \int_{0}^{\infty}\norm{\nabla u}_{L_{\infty}(\R^3)}\,dt
    %&\leq \int_{0}^{\infty}\norm{\nabla^{2} u}^{\frac{p(m-3)}{3(m-p)}}_{L_{p}}\norm{\nabla^{2} u}^{\frac{m(3-p)}{3(m-p)}}_{L_{m}}\,dt\\
    &\leq \int_{0}^{\infty} t^{-\frac{m(3-p)}{3(m-p)}} \norm{\nabla^{2} u}^{\frac{p(m-3)}{3(m-p)}}_{L_{p}(\R^3)}\norm{t\nabla^{2} u}^{\frac{m(3-p)}{3(m-p)}}_{L_{m}(\R^3)}\,dt\\
    &\leq C \norm{\nabla^{2} u}^{\frac{p(m-3)}{3(m-p)}}_{L_{q,1}(\R_+;L_{p}(\R^3))}\norm{t\nabla^{2} u}^{\frac{m(3-p)}{3(m-p)}}_{L_{s,1}(\R_+;L_{m}(\R^3))}\\
    &\leq C \|u_0, \nabla \rho_0\|_{\B^{-1+3/p}_{p,1}(\R^3)}.
    \end{aligned}
\end{equation*}
This finishes the proof of \eqref{eq:tnablau1}.

To get \eqref{eq:tnablau2}, by using \eqref{eq:embed}, we obatin
\begin{equation*}
\begin{aligned}
\int_{0}^{\infty}t\norm{\nabla u}^{2}_{L_{\infty}(\R^d)}\,dt& \leq \int_{0}^{\infty} t\norm{\nabla u}_{\B^{d/m}_{m,1}(\R^d)}\norm{\nabla u}_{L_{\infty}(\R^d)}\,dt\\
&\lesssim \int_{0}^{\infty} \norm{tu}_{\B^{2-2/s}_{m,1}(\R^d)}\norm{\nabla u}_{L_{\infty}(\R^d)}\,dt\\
&\lesssim\norm{tu}_{L_{\infty}(\R_+;\B^{2-2/s}_{m,1}(\R^d))} \norm{\nabla u}_{L_{1}(\R_+;L_{\infty}(\R^d))},
\end{aligned}
\end{equation*}
which along with the inequalities just proved before and Proposition \ref{prop1} ensures \eqref{eq:tnablau2}. 
\medbreak
%\begin{equation}\label{estulilid2}\begin{aligned}
  %\norm{t^{1/2}u}_{L_{\infty}(0,T\times\R^{2})}&\leq \norm{u}^{1/2}_{L_{\infty}(0,T;L_{2})}\norm{tu}^{1/2}_{L_{\infty}(0,T;\B^{1+2/m}_{m,1})}\\
    %&\leq C\norm{u_{0}}^{1/2}_{L_{2}} \exp(C\norm{u_{0}}^{s}_{\B^{-1+2/p}_{p,1}})e^{C\norm{u_{0}}^{2}_{L_{2}}}.
    %\end{aligned}\end{equation}
Finally,  the following  Gagliardo-Nirenberg inequality
\begin{equation*}
    \|z\|^2_{L_\infty(\R^d)}\lesssim \|z\|_{L_d(\R^d)}\|\nabla z\|_{L_\infty(\R^d)} \leq \|z\|_{L_d(\R^d)}\| z\|_{\dot B^{1+2/m}_{m,1}(\R^d)},
\end{equation*}
implies for all $t\in\R_+,$ 
$$t^{1/2} \|u(t)\|_{L_\infty(\R^d)}\lesssim \|u(t)\|_{L_d(\R^d)}^{1/2}\|t u(t)\|_{\dot B^{1+d/m}_{m,1}(\R^d)}^{1/2}$$
whereas embedding \eqref{eq:emdld} and   Proposition \ref{prop0d2}  ensures that \eqref{estulilid2}.
\end{proof}
\smallbreak

By a similar argument, this naturally leads to the following corollary for $\nabla \rho.$ 
\begin{corollary}
Under the same assumption of  Proposition  \ref{prop1}, we have:
\begin{align*}
\int_{0}^{\infty} \norm{\nabla^2 \rho}_{L_{\infty}(\R^d)}\,dt&\leq  C \norm{u_{0},\nabla \rho_0}_{\B^{-1+d/p}_{p,1}(\R^d)}\\
 \biggl(\int_{0}^\infty t\norm{\nabla^2 \rho}^{2}_{L_{\infty}(\R^d)}\,dt\biggr)^{1/2}&\leq  C \norm{u_{0},\nabla \rho_0}_{\B^{-1+d/p}_{p,1}(\R^d)} \\
\underset{t\in \R_+}
\sup ( t)^{1/2} \|\nabla \rho(t)\|_{L_\infty(\R^d)}&\leq C
 \norm{u_{0},\nabla \rho_0}_{\B^{-1+d/p}_{p,1}(\R^d)}.\end{align*}
\end{corollary}
Until now, we have finished all the estimates in Theorem \ref{them1nsk}. In the next section, we will prove the global existence and uniqueness part of Theorem \ref{them1nsk}.

%%%%%%%%%%%%%%%%%%%%%%%%%%%%%%%%%%
\section{Existence and uniqueness}\label{section3}
\subsection{Existence}
Now, for proving rigorously the global existence of a solution on $\R_+$ with the properties
   of Theorem \ref{them1nsk}, one can argue as in \cite{DW}. 
  % Hence we only indicate the main steps and do not give details.  
   In  order to pass to the limit, we shall use the compactness arguments. To overcome the difficulty from the Lorentz spaces $L_{q,1}$ are \emph{nonreflexive} (we  cannot resort to the classical results,  
for example the Aubin-Lions' lemma).
We  consider the approximate
solutions in the following larger (but reflexive)  space 
 $$\W^{2,1}_{p,r}(\R_+ \times \R^d):=
 \bigl\{u\in\cC_b(\R_+;\dot B^{2-2/r}_{p,r}(\R^d)\,:\,
 u_t,\nabla^2u\in L_r(\R_+;L_p(\R^d))\bigr\}
 $$ for some $ 1<r<\infty,$
 then check afterward that the constructed solution has  the desired regularity. % Note that the calculations  just follow from  the properties of the Stokes system and  of the transport equation, basic functional analysis.
 %%%%%%%%%%%%%%%%%%%%%%%%%%%%%%%%%%%%%%

%   we shall smooth out the data so as  to apply prior results ensuring the existence of a sequence $(\rho^n,u^n,\nabla P^n)_{n\in\N}$ of  strong 
%(relatively) smooth solutions to $(INSK)$ (see e.g. ) that satisfy (uniformly) the bounds we proved before.
%The estimates of Sections \ref{section2} 
% will guarantee that  
%the solution $(\rho^n,u^n,\nabla P^n)_{n\in\N}$ is global and  uniformly 
%bounded in the expected spaces. 

 \medbreak
% From now on,  we drop  $\R^d$  in the notations for  the norms. 
Firstly, let us smooth out the initial
 data $\rho_0$ and $u_0$ by means of non-negative mollifiers,  denote $a^n=\rho^n-1$ which is also satisfy the mass equation  of $(INSK),$ to get a  sequence $(\rho_0^n, u_0^n)_{n\in{\mathbb N}}$ of smooth data such that
 %$u^n_{0} \in \B^{-1+d/p}_{p,r}$  for all $1<r<\infty,$ and 
\begin{equation}\label{eq:unifbound}
\norm{a^n_0}_{L_\infty}\leq  \norm{a_0}_{L_\infty},\quad \norm{u^n_{0},~\nabla \rho^n_0}_{\B^{-1+d/p}_{p,1}}\leq C \norm{u_{0},~ \nabla \rho_0}_{\B^{-1+d/p}_{p,1}}
\end{equation}
with, in addition, $$a^{n}_{0}\rightharpoonup a_0  \quad \text{weak * in}\quad L_\infty \andf 
u_{0}^n\to u_{0}, \quad \nabla \rho_{0}^n\to \nabla \rho_{0} \quad \text{strongly   in}\quad \B^{-1+d/p}_{p,1}. $$ 

For the System $(INSK), $
in light of e.g. \cite{BC, JB, BGLRS}, there exists $T>0$ such that  supplemented
with initial data $(\rho_0^n, u_0^n)$ admits a unique smooth local solution  $(\rho^n,u^n,\nabla P^n)$ on $[0,T]\times\R^d.$
In particular, 
$a^n\in\cC_b([0,T]\times\R^d)$ and
$(\nabla \rho^n, u^n,\nabla P^n)$ is in the following space for all $r\geq1$: 
 $$E^{p,r}_{T}=\bigl\{(u,\nabla P)\!\with\!  u\in \W^{2,1}_{p,r}(0,T\times \R^d) \!\andf\! \nabla P \in L_{r}(0,T;L_p)\bigr\}\cdotp$$
Denote $T^n$ is the maximal time of existence  of  $(\rho^n,u^n,\nabla P^n).$ Let $$v^n=u^n+\nabla \log \rho^n, $$ then the triple  $(\nabla \rho^n, v^n, u^n)$ satisfy
\begin{equation}
\label{sys:ruvn}
\left\{\begin{aligned}
&\d_t\nabla \rho^n-\nu\Delta \nabla \rho^n=-\nabla\div(\rho^n v^n),
\\
 &  \d_t u^n-\Delta u^n+\nabla P^n=-(\rho^n-1)\d_t u^n-\rho^n u^n\cdot \nabla u^n+g^n, \\
 &  \d_t v^n-\Delta v^n-\nabla \div v^n=-(\rho^n-1)\d_t v^n-\rho^n v^n\cdot \nabla v^n+\phi^n \\
 &\hspace{5cm}-\nabla P^n+\div((\mu(\rho^n)-\mu(1))D(v^n)),\\
&\div u^n=0,\\
&(\rho^n,u^n)|_{t=0}=(\rho^n_{0},u^n_{0})
\end{aligned}\right.
\end{equation}
with $g^n=\div((\mu(\rho^n)-\mu(1))D(u^n))-\div(k(\rho^n)\nabla \rho^n\otimes \nabla \rho^n),$ and
\begin{multline*}
    \phi^n:=
-\div(k(\rho^n)\nabla \rho^n \otimes \nabla \rho^n)-\nu\nabla \rho^n \cdot Dv^n+\nu\nabla \rho^n\cdot \nabla v^n\\+\nu^2 \nabla \rho^n \cdot \nabla(\nabla \log \rho^n)+
2\nu\div(\mu(\rho^n)\nabla (\nabla \log \rho^n)).
\end{multline*}
 Because of  the calculations of the previous sections  just follow from  the properties of the heat flow, Stokes equation, Lam\'e equation,  transport equation, and basic functional analysis, each $(a^{n}, \nabla \rho^n, u^{n},\nabla P^n)$ satisfies the estimates in $[0, T^n]$,  and thus 
 \begin{equation}\label{eq:anbound}
 \norm{a^{n}(t)}_{L_\infty}= \norm{a^{n}_0}_{L_\infty} \leq \norm{a_0}_{L_\infty} \quad \text{for all} \quad t\in [0, T^n].
 \end{equation}
and, due to \eqref{eq:unifbound}, which gives rise to
 \begin{equation}\label{eq:unlorentz}
 \norm{u^n, \nabla \rho^n, v^n}_{\W^{2,1}_{p,(q,1)}(0,T^n\times\R^d)}+\norm{\nabla P^n}_{L_{q,1}(0,T^n;L_p)}\leq C\norm{u_0,\nabla \rho_0}_{\B^{-1+d/p}_{p,1}}\cdotp
 \end{equation}
Moreover, taking any $r\in(1,\infty)$
and 
applying Proposition 2.1 in \cite{DM2}, Proposition \ref{propregularity} and \ref{pro:mrforlame} with $q=r$~to \eqref{sys:ruvn}.
All the decomposition \eqref{eq:divmr}, \eqref{eq:mr}, \eqref{eq:krho}, \eqref{eq:nnr} and \eqref{eq:ndrv} yields for all $T<T^n,$
\begin{small}
\begin{equation*}
\begin{aligned}
\norm{u^n, v^n/2, \nabla u^n, \nabla P^n/2}_{E^{p,r}_T}\lesssim& \norm{u^n_0, \nabla \rho^n_0}_{\B^{2-2/r}_{p,r}}+\norm{a^n \d_t u^n+(1+a^n) u^n\cdot \nabla u^n}_{L_{r}(0,T;L_p)}\\
&+\norm{a^n \d_t v^n+(1+a^n) v^n\cdot \nabla v^n}_{L_{r}(0,T;L_p)}\\
&+\norm{a^n \Delta u^n}_{L_{r}(0,T;L_p)}+\norm{a^n \Delta v^n}_{L_{r}(0,T;L_p)}\\
&+\norm{\nabla \rho^n\cdot \nabla v^n}_{L_{r}(0,T;L_p)}+\norm{\nabla \rho^n\cdot Dv^n}_{L_{r}(0,T;L_p)}\\
&+\norm{\abs{\nabla \rho^n}^2\cdot \nabla \rho^n}_{L_{r}(0,T;L_p)}+\norm{\nabla \rho^n\cdot \Delta\rho^n}_{L_{r}(0,T;L_p)}\\
&+\norm{\nabla \rho^n\cdot\nabla^2 \rho^n}_{L_{r}(0,T;L_p)}+\norm{\nabla \rho^n\cdot D( u^n)}_{L_{r}(0,T;L_p)}\\
&+\norm{\nabla \rho^n\cdot \div v^n}_{L_{r}(0,T;L_p)}+\norm{v^n\nabla(\nabla \rho^n)}_{L_{r}(0,T;L_p)}\\
&+\norm{\rho^n \nabla^2(\nabla \rho^n)}_{L_{r}(0,T;L_p)}+\norm{\rho^n\nabla \div v^n}_{L_{r}(0,T;L_p)}.
\end{aligned}
\end{equation*}
\end{small}
For the last two terms, we write, using \eqref{eq:esimpot} and $\nabla \div=\q \Delta$ with $\q$ is bound on $L_{r}(0,T;L_p),$
\begin{multline*}
    \norm{\rho^n \nabla^2(\nabla \rho^n)}_{L_{r}(0,T;L_p)}+\norm{\rho^n\nabla \div v^n}_{L_{r}(0,T;L_p)}\\
    \lesssim \norm{\nabla\rho^n }_{L_{\infty}(0,T;\B^{-1+d/p}_{p,1})} ( \norm{\nabla^2(\nabla \rho^n)}_{L_{r}(0,T;L_p)}+\norm{\nabla^2 v^n}_{L_{r}(0,T;L_p)}).
\end{multline*}
For the first three lines, the H\"older inequality yields
\begin{align*}
   & \norm{a^n \d_t u^n+(1+a^n) u^n\cdot \nabla u^n}_{L_{r}(0,T;L_p)}\\\lesssim & \|a^n\|_{L_\infty([0,T]\times \R^d)}\norm{ \d_t u^n}_{L_{r}(0,T;L_p)}+ \|a^n+1\|_{L_\infty([0,T]\times \R^d)}\norm{ u^n\cdot \nabla u^n}_{L_{r}(0,T;L_p)}\\
   & \norm{a^n \d_t v^n+(1+a^n) v^n\cdot \nabla v^n}_{L_{r}(0,T;L_p)}\\\lesssim & \|a^n\|_{L_\infty([0,T]\times \R^d)}\norm{ \d_t v^n}_{L_{r}(0,T;L_p)}+ \|a^n+1\|_{L_\infty([0,T]\times \R^d)}\norm{ v^n\cdot \nabla v^n}_{L_{r}(0,T;L_p)}\\
   & \norm{a^n \Delta u^n}_{L_{r}(0,T;L_p)}+\norm{a^n \Delta v^n}_{L_{r}(0,T;L_p)}\\
   \lesssim& \|a^n\|_{L_\infty([0,T]\times \R^d)}(\norm{\Delta u^n}_{L_{r}(0,T;L_p)}+\norm{\Delta v^n}_{L_{r}(0,T;L_p)}).
\end{align*}

Hence, taking advantage of \eqref{inidr}, \eqref{eq:anbound}, \eqref{eq:unlorentz}, 
the above inequality becomes 
\begin{equation*}
\begin{aligned}
\norm{u^n, v^n/2, \nabla u^n, \nabla P^n/2}_{E^{p,r}_T}\lesssim &\norm{u^n_0, \nabla \rho^n_0}_{\B^{2-2/r}_{p,r}}+\norm{u^n\cdot \nabla u^n}_{L_{r}(0,T;L_p)}+\norm{ v^n\cdot \nabla v^n}_{L_{r}(0,T;L_p)}\\
&+\norm{\nabla \rho^n\cdot \nabla v^n}_{L_{r}(0,T;L_p)}+\norm{\nabla \rho^n\cdot \div v^n}_{L_{r}(0,T;L_p)}\\
&+\norm{\nabla \rho^n\cdot Dv^n}_{L_{r}(0,T;L_p)}+\norm{\nabla \rho^n\cdot D(u^n)}_{L_{r}(0,T;L_p)}\\
&+\norm{\abs{\nabla \rho^n}^2\cdot \nabla \rho^n}_{L_{r}(0,T;L_p)}+\norm{\nabla \rho^n\cdot \Delta\rho^n}_{L_{r}(0,T;L_p)}\\
&+\norm{\nabla \rho^n\cdot\nabla^2 \rho^n}_{L_{r}(0,T;L_p)}+\norm{v^n\nabla(\nabla \rho^n)}_{L_{r}(0,T;L_p)}.
\end{aligned}
\end{equation*}
Thus, defining $\beta$ by  ${1}/{\beta}+{1}/{d}-{1}/{dr}={1}/{p}$ and using again the H\"older inequality, we get
\begin{small}
  \begin{align*}
     &\norm{u^n\cdot \nabla u^n}_{L_p}\lesssim  \norm{u^n}_{L_{\frac{dr}{r-1}}}\norm{\nabla u^n}_{L_{\beta}}, &\norm{v^n\cdot \nabla v^n}_{L_p}\lesssim  \norm{v^n}_{L_{\frac{dr}{r-1}}}\norm{\nabla v^n}_{L_{\beta}},\\
     &\norm{\nabla \rho^n \cdot \nabla v^n}_{L_p}\lesssim  \norm{\nabla \rho^n}_{L_{\frac{dr}{r-1}}}\norm{\nabla v^n}_{L_{\beta}},&\norm{\nabla \rho^n\cdot \div v^n}_{L_p}\lesssim  \norm{\nabla \rho^n}_{L_{\frac{dr}{r-1}}}\norm{\div v^n}_{L_{\beta}}, \\
     &\norm{\nabla \rho^n\cdot Dv^n}_{L_p}\lesssim  \norm{\nabla \rho^n}_{L_{\frac{dr}{r-1}}}\norm{\nabla v^n}_{L_{\beta}},& \norm{\nabla \rho^n \cdot D(u^n)}_{L_p}\lesssim  \norm{\nabla \rho^n}_{L_{\frac{dr}{r-1}}}\norm{ D( u^n)}_{L_{\beta}},\\
       &\norm{\abs{\nabla \rho^n}^2\cdot \nabla \rho^n}_{L_p}\lesssim  \norm{\nabla \rho^n}_{L_{\frac{dr}{r-1}}}\norm{|\nabla \rho^n|^2}_{L_{\beta}},&\norm{\nabla \rho^n \cdot \Delta \rho^n}_{L_p}\lesssim  \norm{\nabla \rho^n}_{L_{\frac{dr}{r-1}}}\norm{\nabla^2 \rho^n}_{L_{\beta}},\\
        &\norm{\nabla \rho^n\cdot \nabla^2 \rho^n}_{L_p}\lesssim  \norm{\nabla \rho^n}_{L_{\frac{dr}{r-1}}}\norm{\nabla^2 \rho^n}_{L_{\beta}},&  \norm{v^n \cdot \nabla(\nabla  \rho^n)}_{L_p}\lesssim  \norm{v^n}_{L_{\frac{dr}{r-1}}}\norm{\nabla^2 \rho^n}_{L_{\beta}}.
 \end{align*}
 \end{small}
As a consequence, one can obtain
 \begin{equation*}
 \begin{aligned}
\norm{u^n, v^n/2, \nabla u^n, \nabla P^n/2}^r_{E^{p,r}_T} 
 \lesssim &\norm{u^n_0}^{r}_{\B^{2-2/r}_{p,q}}+\int_{0}^{T} (\norm{u^n}^{r}_{L_{\frac{dr}{r-1}}}+\norm{\nabla \rho^n}^{r}_{L_{\frac{dr}{r-1}}})\norm{\nabla u^n}^{r}_{L_{\beta}}\,dt\\&+\int_{0}^{T} (\norm{v^n}^{r}_{L_{\frac{dr}{r-1}}}+\norm{\nabla \rho^n}^{r}_{L_{\frac{dr}{r-1}}})(\norm{\nabla v^n}^{r}_{L_{\beta}}+\norm{\nabla^2 \rho^n}^{r}_{L_{\beta}})\,dt\\
&\qquad\qquad +\int_{0}^{T}\norm{\nabla \rho^n}^r_{L_{\frac{dr}{r-1}}}\norm{\nabla \rho^n}^{2r}_{L_{2\beta}}\,dt.
 \end{aligned}
 \end{equation*}
To deal with the last term of the above inequality, we use Proposition \ref{interpolation} to get
\begin{align*}
    &\|z\|_{\B^{d(\frac{1}{p}-\frac{1}{2\beta})}_{p,1}}\lesssim \|z\|^{1/2}_{\B^{-1+\frac dp}_{p,\infty}}\|z\|^{1/2}_{\B^{1+\frac dp-\frac d\beta}_{p,\infty}},\\
   & \|z\|_{\B^{1+\frac dp-\frac d\beta}_{p,\infty}}\lesssim \|z\|^{1/2}_{\B^{2}_{p,\infty}}\|z\|^{1/2}_{\B^{2-\frac 2r}_{p,\infty}},
\end{align*}
which along with embedding
$\B^{d(1/p-1/2\beta)}_{p,1}\hookrightarrow \B^0_{2\beta,1}\hookrightarrow L_{2\beta}$ implies
$$\int_{0}^{T}\norm{\nabla \rho^n}^r_{L_{\frac{dr}{r-1}}}\norm{\nabla \rho^n}^{2r}_{L_{2\beta}}\,dt\lesssim \int_{0}^{T}\norm{\nabla \rho^n}^r_{L_{\frac{dr}{r-1}}}\norm{\nabla \rho^n}^r_{\B^{-1+d/p}_{p,1}}\norm{\nabla \rho^n}^{\frac r2}_{\B^{2-2/r}_{p,r}} \norm{\nabla^2(\nabla \rho^n)}^{\frac r2}_{L_p}\,dt.$$
% In the beginning, we want to prove that $\|z\|^2_{L_{2\beta}}\leq \|z\|_{L_{d}}\|\nabla z\|_{L_{\beta}}$, but we are not sure that whether $2\beta>d.$
As a consequence, Proposition \ref{prop:for existence} and Young's inequality gives for all $\eps>0,$
\begin{small}
\begin{multline*}
  \norm{u^n, v^n/2, \nabla u^n, \nabla P^n/2}^r_{E^{p,r}_T}\leq
C\norm{u^n_0,\nabla \rho^n}^r_{\B^{2-2/r}_{p,r}} +\eps \int_0^T 
    \norm{\nabla^{2} u^n, ~\nabla^{2} v^n, ~\nabla^{2} (\nabla \rho^n) }^{r}_{L_p}\,dt\\+C_\eps \biggl(\int_0^T (\norm{u^n}^{2r}_{L_{\frac{dr}{r-1}}}+\norm{\nabla \rho^n}^{2r}_{L_{\frac{dr}{r-1}}})
    \norm{u^n}^{r}_{\B^{2-2/r}_{p,r}}\,dt\\
    +\int_0^T (\norm{v^n}^{2r}_{L_{\frac{dr}{r-1}}}+\norm{\nabla \rho^n}^{2r}_{L_{\frac{dr}{r-1}}})
    (\norm{v^n}^{r}_{\B^{2-2/r}_{p,r}}+\norm{\nabla \rho^n}^{r}_{\B^{2-2/r}_{p,r}})\,dt\\
    +\norm{\nabla \rho^n}^{2r}_{L_\infty(0,T;\B^{-1+d/p}_{p,1})} \int_{0}^{T}\norm{\nabla \rho^n}^{2r}_{L_{\frac{dr}{r-1}}}\norm{\nabla \rho^n}^{r}_{\B^{2-2/r}_{p,r}} \,dt\biggr).
\end{multline*}
\end{small}

Then, taking $\eps$ small enough and using Gronwall's inequality yields
\begin{equation}\label{eq:unest}
 \norm{u^n, v^n/2, \nabla u^n, \nabla P^n/2}^r_{E^{p,r}_T}\leq C \norm{u^n_0, \nabla \rho^n_0}^r_{\B^{2-2/r}_{p,r}} \exp(Ch(T))\cdotp
\end{equation}
with $$h(T)=\int_0^T (\norm{u^n}^{2r}_{L_{\frac{dr}{r-1}}}+\norm{v^n}^{2r}_{L_{\frac{dr}{r-1}}}) \,dt+(1+\norm{\nabla \rho^n}^{2r}_{L_\infty(0,T;\B^{-1+d/p}_{p,1})})\int_0^T \norm{\nabla \rho^n}^{2r}_{L_{\frac{dr}{r-1}}}\,dt .$$

In the end, Gagliardo-Nirenberg inequality and embedding give
\begin{equation}
    \label{eq:inldr}
    \norm{z}_{L_{\frac{dr}{r-1}}}\leq \norm{z}^{\frac 1r}_{L_{\infty}}\norm{z}^{1-\frac 1r}_{\B^{-1+d/p}_{p,1}},
\end{equation}
which implies that
\begin{equation}\label{eq:new}
\begin{aligned}
   & \int_0^T \norm{u^n}^{2r}_{L_{\frac{dr}{r-1}}} \,dt 
\leq \norm{u^n}^2_{L_2(0,T;L_\infty)}\norm{u^n}^{2(r-1)}_{L_\infty(0,T;\B^{-1+d/p}_{p,1})},\\
&\int_0^T \norm{v^n}^{2r}_{L_{\frac{dr}{r-1}}} \,dt 
\leq \norm{v^n}^2_{L_2(0,T;L_\infty)}\norm{v^n}^{2(r-1)}_{L_\infty(0,T;\B^{-1+d/p}_{p,1})},\\
& \int_0^T \norm{\nabla \rho^n}^{2r}_{L_{\frac{dr}{r-1}}} \,dt 
\leq \norm{\nabla \rho^n}^2_{L_2(0,T;L_\infty)}\norm{\nabla \rho^n}^{2(r-1)}_{L_\infty(0,T;\B^{-1+d/p}_{p,1})}.
\end{aligned}
\end{equation}

Now, we deduce from 
 Proposition \ref{prop0d2}
 that the right-hand side of \eqref{eq:new}  are bounded by $\|u_0, \nabla \rho_0\|_{\dot B^{-1+d/p}_{p,1}}.$
 Hence,  coming back to \eqref{eq:unest} and 
 using a classical continuation argument ensures that the solution is global, and the approximate solution such that $$(\nabla \rho^n,  u^n) \in W^{2,1}_{p,r}(\R_+\times\R^d)\times W^{2,1}_{p,r}(\R_+\times\R^d)$$ with $1<r<\infty.$
Otherwise, owing to the solution is smooth and
 \eqref{eq:unifbound} is satisfied, all the a priori  estimates
 of Sections \ref{section2}
 are satisfied uniformly with respect to $n.$
 
In addition, 
$(u^{n}, \nabla P^n)_{n\in\N}$ is  bounded in 
$E^{p,q}_T$  and $ \nabla \rho^n$ is bounded in $ W^{2,1}_{p,r}([0,T]\times\R^d)$ for all $T\geq0.$
 This,  along with   \eqref{eq:anbound} and \eqref{eq:unlorentz} guarantee that 
 there exists a subsequence, still denoted by $(a^{n},\nabla \rho^n, u^{n},\nabla P^{n})_{n\in{\mathbb N}},$ and $(a,\nabla \rho, u,\nabla P)$ with
 $$a\in L_{\infty}(\R_{+}\times\R^{d}),\quad \nabla P\in L_{q}(\R_{+};L_{p}(\R^{d}))\andf u,~ \nabla \rho\in \W^{2,1}_{p,q}(\R_{+}\times \R^d)$$
 such that 
 \begin{equation}\label{eq:weak limit}
 \begin{aligned}
 &a^{n}\rightharpoonup a\quad \text{weak *  in } \quad L_{\infty}(\R_{+}\times\R^{d}),\\
  &u^{n}\rightharpoonup u\quad \text{weak *  in } \quad L_{\infty}(\R_{+};\dot B^{2-2/q}_{p,q}),\\
 &(\d_{t}u^{n}, \nabla^{2}u^{n})\rightharpoonup (\d_{t}u, \nabla^{2}u)\quad  \text{weakly in} \quad 
 L_{q}(\R_{+};L_{p}),\\
 &\nabla P^{n} \rightharpoonup \nabla P\quad  \text{weakly in} \quad  
 L_{q}(\R_{+};L_{p}),\\
 &\nabla \rho^{n}\rightharpoonup \nabla \rho\quad \text{weak *  in } \quad L_{\infty}(\R_{+};\dot B^{2-2/q}_{p,q}),\\
 &(\d_{t}\nabla \rho^{n}, \nabla^{2}\rho^{n})\rightharpoonup (\d_{t}(\nabla \rho), \nabla^{2}(\nabla \rho))\quad  \text{weakly in} \quad 
 L_{q}(\R_{+};L_{p}),\\
 \end{aligned}
 \end{equation}
% Since our proof is valid for all $r\in(1,\infty),$ one 
% can take $r=q_1$ or $r=q_2$ with $1<q_1<q<q_2<\infty$ such that
% $\frac1q=\frac12(\frac1{q_1}+\frac1{q_2})\cdotp$
% Hence,  using the following results of interpolation:
% $$\W^{2,1}_{p,(q,1)}(\R_{+}\times %\R^d)=(\W^{2,1}_{p,q_1}(\R_{+}\times %\R^d),\W^{2,1}_{p,q_2}(\R_{+}\times \R^d))_{\frac12,1},$$ 
% and (see  \cite[Th2:1.18.6]{HT}): %$$L_{q,1}(\R_{+};L_p)=(L_{q_1}(\R_+;L_p), %L_{q_2}(\R_+;L_p))_{\frac12,1},$$
In terms of the Fatou property, it holds for
 all the spaces under consideration 
in the previous sections, the estimates proved therein
as still valid. For example, 
one gets
  $$\nabla \rho,~ u\in \W^{2,1}_{p,(q,1)}(\R_{+}\times \R^d)\quad \text{and} \quad \nabla P\in L_{q,1}(\R_+;L_p (\R^d)).$$
Note that the fact  that $(\d_t u^n, \d_t (\nabla \rho^n))_{n\in {\mathbb N}}$ is bounded in $L_q(\R_{+};L_p)$, so that we can use  
 the   Arzel\`a-Ascoli Theorem 
 to get strong convergence 
results for $u, \nabla \rho$ like, 
for instance, for all small enough $\varepsilon>0,$  
\begin{equation}\label{eq:strong limit}
\begin{aligned}
&u^{n}\rightarrow u \ \text{strongly in}\  L_{\infty,loc}(\R_{+};L_{d-\varepsilon,loc}(\R^d)),\\
&\nabla u^{n}\rightarrow \nabla u \ \text{strongly  in }\ L_{q,loc}(\R_{+};L_{p^*-\varepsilon,loc})\with\frac1{p^*}
=\frac1p-\frac1d\\
&\nabla \rho^{n}\rightarrow \nabla \rho \ \text{strongly in}\  L_{\infty,loc}(\R_{+};L_{d-\varepsilon,loc}(\R^d)),\\
&\nabla^2 \rho^{n}\rightarrow \nabla^2 \rho\ \text{strongly  in }\ L_{q,loc}(\R_{+};L_{p^*-\varepsilon,loc})\with\frac1{p^*}
=\frac1p-\frac1d
\cdotp\end{aligned}
\end{equation}
This allows to pass to the limit in 
the convection term and  the capillary term of the momentum 
equation of $(INSK)$. On the other hand, 
to pass
to the limit in the terms containing $a^n$
and conclude that $(a, \nabla \rho, u,\nabla P)$ is a global weak solution. 
Since  $a\in L_\infty(\R_+\times\R^d),$ $\nabla^2 \rho\in L_{2q}(\R_{+};L_{\frac{dq}{2q-1}}),$ $\nabla u\in L_{2q}(\R_{+};L_{\frac{dq}{2q-1}})$ and $\div u=0,$
the Di Perna - Lions
theory in \cite{DL1989} ensures that  $a$ is the only solution to
the mass  equation of $(INSK)$ and  that
$$a^{n}\rightarrow a \ \text{strongly  in }\ L_{\alpha,loc}(\R_{+}\times \R^d)\quad\hbox{for all }\ 1<\alpha<\infty.$$ 
 
   Finally, using the strong convergence results and passing to the limit in the 
momentum equation of $(INSK),$ we can complete the proof of the global existence of $(INSK).$

\medbreak
\subsection{Uniqueness}
To prove the uniqueness of the solution,  it is easier to proceed as C. Burtea and F. Charve
in \cite{BC} and to rewrite the system in Lagrangian coordinates.
%as mentioned before,  we can consider it in the Euler coordinates, to bound the differences $(\delta \rho, \delta u)$ in space $H^{-1}\times L_2.$ On the other hand, 
So we introduce the flow associated with $u$ which is the solution to
\begin{equation*}
\left\{\begin{aligned}
&\frac{d}{dt}X(t,y)=u(t,X(t,y)) \\
 &X|_{t=0}=y.
\end{aligned}\right.
\end{equation*} 
As the velocity is divergence-free, the jacobian determinant satisfies $\det(D_y X)\equiv 1.$  In Lagrangian coordinates $(t,y),$ a solution $(\rho , u, P)$ to $(INSK)$ is recast in $(\eta, v, Q)$ as
$$\eta(t,y):=\rho(t,X(t,y)),\!\quad\! v(t,y):=u(t,X(t,y)) \andf Q(t,y):=P(t,X(t,y))$$  and the triplet satisfies
\begin{equation}\label{sys:inskl}
\left\{\begin{aligned}
&\d_t \eta=0, \\
 & \eta \d_tv-\div(\mu(\eta)A^t D_A(v))+A\nabla G=-\div(k(\eta)A^t(A \nabla \eta)\otimes (A \nabla \eta) )\\
&\div(A v)=0\\
&(\eta,v)|_{t=0}=(\rho_{0},u_{0}) 
\end{aligned}\right.
\end{equation}
where $A=(\nabla_y X)^{-1}$ and $D_{A}(v)=Dv\cdot A^t+A \cdot \nabla v.$ Based on the prior estimates got before, see Corollary \ref{coro1d2},  we find that  the solution of $(INSK)$ construct in Theorem \ref{them1nsk} satisfies  
$$\int_{0}^{\infty} \norm{\nabla u}_{L_{\infty}}\,dt\leq  C \norm{u_{0},\nabla \rho_0}_{\B^{-1+d/p}_{p,1}}<1$$
which implies \eqref{sys:inskl} is equivalent to $(INSK).$  Hence the uniqueness can be done in the Lagrangian coordinates, see e.g. \cite{CB2017}. As the density becomes constant, we just need to consider the differences $(\delta u, \delta P)$    in space $ E^{p}_{q,1}$  with
$$E^{p}_{q,r}=\bigl\{(u,\nabla P)\!\with\!  u\in \W^{2,1}_{p,(q,r)}(\R_+\times \R^d) \!\andf\! \nabla P \in L_{q, r}(\R_+;L_p)\bigr\}\cdotp
$$

\subsection* {Acknowledgments:}
 %The authors  are  indebted to the anonymous referees for  their careful reading and relevant suggestions.
The author acknowledges funding from the Italian Ministry of University and Research, project PRIN 2022HSSYPN. The author would like to thank her PhD thesis 
supervisor Prof. Rapha\"el Danchin.

\appendix
\section{}
Here, we present some results of 
Besov spaces and Lorentz spaces, 
prove maximal regularity estimates 
in Lorentz spaces for \eqref{eq:stokes}, \eqref{sys:lema}. The following results was presented in \cite{DW}, so we omit the details of proofs here.

\medbreak
The following properties of Lorentz spaces may be found in \cite{LG, DW}.
(and  \cite[Th2:1.18.6]{HT} for the first item):
\begin{proposition}\label{p:lorentz}
For Lorentz space, we have the following properties:
\begin{enumerate}
\item {\rm Interpolation}: For all $1\leq r,q\leq \infty$ and $\theta\in (0,1)$, we have 
$$\left(L_{p_{1}}(\R_{+};L_{q}(\R^{d}));L_{p_{2}}(\R_{+};L_{q}(\R^{d}))\right)_{\theta,r}=L_{p,r}(\R_{+};L_{q}(\R^{d}))),$$
where $1<p_{1}<p<p_{2}<\infty$ are such that  $\frac{1}{p}=\frac{(1-\theta)}{p_{1}}+\frac{\theta}{p_{2}}\cdotp$
\item {\rm Embedding}: $L_{p,r_{1}}\hookrightarrow L_{p,r_{2}} \ \text{if}\  r_{1}\leq r_{2},$ and $L_{p,p}=L_{p}.$ 
\item {\rm H\"older inequality}: for $1<p,p_{1},p_{2}<\infty$ and $1 \leq r,r_{1},r_{2}\leq \infty,$
 we have $$\norm{fg}_{L_{p,r}}\lesssim \norm{f}_{L_{p_{1},r_{1}}}\norm{g}_{L_{p_{2},r_{2}}}\quad\hbox{if}\quad \frac{1}{p}=\frac{1}{p_{1}}+\frac{1}{p_{2}}\andf\frac{1}{r}=\frac{1}{r_{1}}+\frac{1}{r_{2}}\cdotp$$ 
 This still holds for couples $(1,1)$ and $(\infty,\infty)$ with the convention $L_{1,1}=L_{1}$ and $L_{\infty,\infty}=L_{\infty}.$
\item For any $\alpha>0$ and nonnegative measurable function $f,$ we have  $\norm{f^{\alpha}}_{L_{p,r}}=\norm{f}^{\alpha}_{L_{p\alpha,r\alpha}}$.
\item For any  $k>0$, we have $\norm{x^{-k}1_{\R_+}}_{L_{1/k,\infty}}=1.$ 
\end{enumerate}
\end{proposition}
Next, let us  state a few classical properties of Besov spaces.
\begin{proposition}[Besov embedding]\label{p:A1} The following embedding for Besov space holds, \begin{enumerate}
\item  For any $(p,q)$ in $[1,\infty]^{2}$ such that $p\leq q,$ we have
 $$\B^{d/p-d/q}_{p,1}(\R^{d})\hookrightarrow L_{q}(\R^{d}).$$
\item  Let $1\leq p_{1}\leq p_{2}\leq \infty$ and $1\leq r_{1}\leq r_{2}\leq \infty.$ Then, for any real number $s$, 
$$\B^{s}_{p_{1},r_{1}}(\R^{d})\hookrightarrow \B^{s-d(\frac{1}{p_{1}}-\frac{1}{p_{2}})}_{p_{2},r_{2}}(\R^{d}).
%\andf B^{s}_{p_{1},r_{1}}(\R^{d})\hookrightarrow %B^{s-d(\frac{1}{p_{1}}-\frac{1}{p_{2}})}_{p_{2},r_{2}}(\R^{d}). 
$$
\end{enumerate}
\end{proposition}
The Bernstein inequality (see \cite[Lemma 2.1.]{BCD}) enables us to establish the following useful property.
\begin{proposition}
  \label{dervieq}
  Let $(s,p,r)$ be in $\R\times [1,\infty]^2.$ Let $u$ be in $\B^s_{p,r}(\R^d),$ then we have $\nabla u $ in $\B^{s-1}_{p,r}(\R^d)$ and there exist a position constant $C$ such that
  $$ C^{-1}\|u\|_{\B^s_{p,r}(\R^d)}\leq \|\nabla u\|_{\B^{s-1}_{p,r}(\R^d)}\leq C\|u\|_{\B^s_{p,r}(\R^d)}.$$
\end{proposition}
The interpolation theory in  Besov spaces played  an important role in our  paper.  
Below are listed some results that we used (see details in \cite[Prop. 2.22]{BCD} or in 
 \cite[chapter 2.4.2]{HT}). 
\begin{proposition}[Interpolation]\label{interpolation}
A constant $C$ exists that  satisfies the following properties. If $s_{1}$ and $s_{2}$ are real numbers such that $s_{1}<s_{2}$
and $\theta\in ]0,1[,$ then we have, for any $(p,r)\in [1,\infty]^{2}$ and any tempered distribution $u$ satisfying \eqref{eq:lf}, 
$$\norm{u}_{\B^{\theta s_{1}+(1-\theta)s_{2}}_{p,r}(\R^{d})}\leq \norm{u}^{\theta}_{\B^{s_{1}}_{p,r}}\norm{u}^{1-\theta}_{\B^{s_{2}}_{p,r}(\R^{d})} $$
and, for some constant $C$ depending only on $\theta$ and $s_2-s_1,$ 
$$\norm{u}_{\B^{\theta s_{1}+(1-\theta)s_{2}}_{p,1}(\R^{d})}\leq
C\norm{u}^{\theta}_{\B^{s_{1}}_{p,\infty}(\R^{d})}\norm{u}^{1-\theta}_{\B^{s_{2}}_{p,\infty}(\R^{d})}. $$
Furthermore,  we have  for all $s\in(0,1)$ and $(p,q)\in[1,\infty]^2:$
$$\B^{s}_{p,q}(\R^d)=\bigl(L_{p}(\R^d);\W^{1}_{p}(\R^d)\bigr)_{s,q}.$$
\end{proposition}

The following proposition %in the spirit of \cite[Theorem 2.1]{RD},
has been used several times, see e.g. \cite{DW}.  
\begin{proposition}\label{prop:for existence}
Let $1\leq q<\infty,$ $1\leq p<r\leq\infty$ and $\theta\in(0,1)$ such that
\begin{equation}\label{eq:q}
 \frac{1}{r}+\frac{1}{d}-\frac{2\theta}{dq}=\frac{1}{p}\cdotp
\end{equation}
 Then, there exists $C$ so that the following inequality holds true
 $$\norm{\nabla u}_{L_r(\R^d)}\leq C\norm{\nabla^2 u}^{\theta}_{L_p(\R^d)}\norm{u}^{1-\theta}_{\B^{2-2/q}_{p,\infty}(\R^d)}\cdotp$$
 \end{proposition}
 
 The following results played a key role in  Sections \ref{section2} which is an easy 
 adaptation of \cite[Prop. 2.1]{DM2} .  Note that estimates in the same spirit \emph{but for source
 terms valued in Besov spaces} have been proved 
 by H. Kozono and S.~Shimizu in \cite{KS}. 
\begin{proposition}\label{propregularity}
Let $1<p,q< \infty$ and $1\leq r\leq \infty.$ Then, for any $u_{0}\in \B^{2-2/q}_{p,r}(\R^{d})$ with $\div u_0=0,$ and any $f\in L_{q,r}(\R_+;L_{p}(\R^{d})),$ the  Stokes system
\eqref{eq:stokes}
has a unique solution $(u,\nabla P)$ with $\nabla P\in L_{q,r}(\R_+;L_p(\R^d))$
and\footnote{Only weak continuity holds if $r=\infty.$}  
$u$ in the space $\W^{2,1}_{p,(q,r)}(\R_+\times \R^{d})$ defined by
$$\bigl\{u\in \mathcal{C}(\R_+;\B^{2-2/q}_{p,r}( \R^{d})):u_{t}, \nabla^{2}u\in L_{q,r}(\R_+;L_{p}( \R^{d})) \bigr\}\cdotp$$
Furthermore, there exists a constant
$C$ such that
\begin{multline}\label{eq:maxreg1}
\mu^{1-1/q}\norm{u}_{L_{\infty}(\R_+;\B^{2-2/q}_{p,r}(\R^{d}))}+\norm{u_{t}, \mu\nabla^{2}u,\nabla P}_{L_{q,r}(\R_+;L_{p}(\R^{d}))}
\\
\leq C\bigl(\mu^{1-1/q}\norm{u_{0}}_{\B^{2-2/q}_{p,r}(\R^{d})}+\norm{f}_{L_{q,r}(\R_+;L_{p}(\R^{d}))}\bigr)\cdotp\end{multline}
Let $\wt s>q$ be such that 
$$\frac1q-\frac1{\wt s}\leq \frac12\andf \frac d{2p}+\frac1q-\frac1{\wt s}>\frac12,$$
 and define $\wt m\geq p$ by the relation
$$\frac{d}{2\wt m}+\frac{1}{\wt s}=\frac{d}{2p}+\frac{1}{q}-\frac{1}{2}\cdotp$$
Then, the following inequality holds true:  
\begin{multline}\label{eq:maxreg2}\mu^{1+\frac{1}{\wt s}-\frac{1}{q}}\norm{\nabla u}_{L_{\wt s,r}(\R_+;L_{\wt m}(\R^{d}))}\\\leq C(\mu^{1-1/q}\norm{u}_{L_{\infty}(\R_+;\B^{2-2/q}_{p,r}(\R^{d}))}+\norm{u_{t}, \mu\nabla^{2}u}_{L_{q,r}(\R_+;L_{p}(\R^{d}))}).
\end{multline}
Finally,  if $2/q+d/p>2,$ then for all $s\in(q,\infty)$ and $m\in(p,\infty)$ such that 
$$\frac{d}{2m}+\frac{1}{s}=\frac{d}{2p}+\frac{1}{q}-1,$$ it holds that 
\begin{multline}\label{eq:maxreg3}\mu^{1+\frac{1}{s}-\frac{1}{q}}\norm{u}_{L_{s,r}(\R_+;L_{m}(\R^{d}))}\\\leq C\bigl(\mu^{1-1/q}\norm{u}_{L_{\infty}(0,T;\B^{2-2/q}_{p,r}(\R^{d}))}+\norm{u_{t}, \mu\nabla^{2}u}_{L_{q,r}(\R_+;L_{p}(\R^{d}))}\bigr)\cdotp
\end{multline}
\end{proposition}

\bigbreak
%\begin{remark}
   %  We can get some information for the pressure $P$ from the momentum equation of $(INSK).$
%\end{remark}
%\begin{remark}
%Unfortunately,  it is not possible to combine maximal regularity estimates
%both for the density and the (modified) velocity so as to prove optimal well-posedness statements.

%As the definition of $v,$ it is natural to consider $\nabla \rho$ directly. From \eqref{eq:rh}, we have
%\begin{equation}
%\label{eq:nr}
%   \d_t (\nabla\rho)-\nu\Delta (\nabla \rho)+\nabla\div(\rho v)=0.
%\end{equation}
%But we find that the maximal regularity property is not suitable to use to consider \eqref{eq:nr}, \eqref{eq:nskeqvc1} and \eqref{eq:u} together, because of the source term of \eqref{eq:nr} and the term $\div(\mu(\rho)\nabla(\nabla \log \rho))$ in $f$.
%For example, if choose $\mu(\rho)=\rho,$ then we have
%$$\div(\rho\nabla(\nabla \log \rho))=-\div(\nabla\rho\otimes\nabla \log \rho)+\Delta \nabla \rho.$$
%Hence, the pressure term becomes $G:=P-2\nu\Delta \rho$. Note that $\div v$ is not equal to zero, we should not use the maximal regularity property to deal with \eqref{eq:nskeqvc1}.
%\end{remark}
%\begin{remark}
%In the definition of $v,$ we should not use $a:=\rho-1$ to replace $\rho.$ 
%\end{remark}

\bigbreak
\section{}
In order to check whether we can get the global existence for large initial data. In the following, we use standard regularity estimates to deal with the  system $(INSK).$
%One way to solve it is to proceed as C. Burtea and F. Charve
%in \cite{BC} and to rewrite the system in Lagrangian coordinates.
%as mentioned before,  we can consider it in the Euler coordinates, to bound the differences $(\delta \rho, \delta u)$ in space $H^{-1}\times L_2.$ On the other hand, 
%So we introduce the flow associated with $u$ which is the solution to
%\begin{equation*}
%\left\{\begin{aligned}
%&\frac{d}{dt}X(t,y)=u(t,X(t,y)) \\
 %&X|_{t=0}=y.
%\end{aligned}\right.
%\end{equation*} 
%As the velocity is divergence-free, the jacobian determinant satisfies $\det(D_y X)\equiv 1.$  In Lagrangian coordiantes $(t,y),$ a solution $(\rho , u, P)$ to $(INSK)$ is recast in $(\eta, v, Q)$ as
%$$\eta(t,y):=\rho(t,X(t,y)),\!\quad\! v(t,y):=u(t,X(t,y)) \andf Q(t,y):=P(t,X(t,y))$$  and the triplet satisfies
%\begin{equation}\label{sys:inskl}
%\left\{\begin{aligned}
%&\d_t \eta=0, \\
% & \eta \d_tv-\div(\mu(\eta)A^t D_A(v))+A\nabla G=-\div(k(\eta)A^t(A \nabla \eta)\otimes (A \nabla \eta) )\\
%&\div(A v)=0\\
%&(\eta,v)|_{t=0}=(\rho_{0},u_{0}) 
%\end{aligned}\right.
%\end{equation}
%where $A=(\nabla_y X)^{-1}$ and $D_{A}(v)=Dv\cdot A^t+A \cdot \nabla v.$ 

%%Here, to establish well-posedness results we shall use 

%\subsection{Estimate of density}
%\numberwithin{equation}{subsubsection}

Here we establish the estimates for the density based on the regularity estimates of the transport equation and heat equation. They are listed in the following proposition:
\begin{proposition}\label{prop:r1}
    Assume that $u$ is divergence free and  that $(\rho, u)$ satisfies the mass equation $\rho_t+\div(\rho u)=0.$  Let the  initial data $\rho_0$ satisfy for some $1\leq p\leq \infty$:
\begin{equation*}
     \nabla\rho_0\in L_\infty,\quad
    \nabla\rho_0 \in L_p \andf \rho_0,\, 1/\rho_0 \in L_{\infty}.
\end{equation*}
Additionally,   suppose that  $\nabla u \in L_1(\R_+).$
Then, for all $t\in \R_+,$  we have  
$\rho(t)$ and $1/\rho(t)$ in $ L_{\infty},$ $\nabla \rho(t)$ in $L_p,$ with 
\begin{align}
    \label{eq:rlin}&\|\rho(t)-1\|_{L_\infty}= \|\rho_0-1\|_{L_\infty} \andf \|1/\rho(t)-1\|_{L_\infty}= \|1/\rho_0-1\|_{L_\infty},\\
    \label{eq:nrlpnew}&\|\nabla \rho(t)\|_{L_p}\leq\|\nabla \rho_0\|_{L_p} \exp{\biggl(\int_0^t \|\nabla u\|_{L_\infty}\,d\tau\biggr)}.
\end{align}
%Moreover, for all $1\leq p\leq %\infty,$ $t\in [0,T]$ we have
  % \begin{equation}\label{eq:twenrlp}
   %    \|t\nabla \rho(t)\|_{L_p}\leq 2 \|\nabla \rho_0\|_{L_p}T.
 %  \end{equation}
\end{proposition}

\begin{proof}
As the vector field $u$ is divergence-free, the mass equation implies
$$(1/\rho)_t+u\cdot \nabla (1/\rho)=0 \andf (\rho-1)_t+u\cdot \nabla (\rho-1)=0$$
which gives \eqref{eq:rlin}. 
\smallbreak
Next, after applying $\nabla$  to the mass equation, we find that $\nabla \rho$  satisfy the following transport equations with source terms:
\begin{equation}
     \label{eq:nr} (\nabla \rho)_t+u\cdot \nabla(\nabla \rho)=-\nabla u\cdot \nabla \rho.
\end{equation}

As the vector field $u$ is divergence free, we get for $t\in \R_+$
$$ \|\nabla \rho(t)\|_{L_p}\leq \|\nabla \rho_0\|_{L_p}+\int_0^t \|\nabla u\cdot \nabla \rho\|_{L_p}\,d\tau \leq \|\nabla \rho_0\|_{L_p}+\int_0^t \|\nabla u\|_{L_\infty}\| \nabla \rho\|_{L_p}\,d\tau
$$
which along with Gronwall inequality   gives
\begin{equation*}
    \|\nabla \rho(t)\|_{L_p}\leq \|\nabla \rho_0\|_{L_p}\exp{\biggl(\int_0^t \|\nabla u\|_{L_\infty}\,d\tau\biggr)}.
\end{equation*}
This completes the proof.

%Finally, multiplying both sides of %\eqref{eq:nrlpnew} by $t$ readily gives \eqref{eq:twenrlp}.
%multiplying both side of \eqref{eq:nr} , we obtain
 %   $$(t\nabla \rho)_t+u\cdot \nabla(t\nabla \rho)+t\nabla u\cdot \nabla \rho=\nabla \rho.$$
  %  Hence, noting the fact that the velocity is divergence free, and using  Gronwall inequality we find that
%$$\|t\nabla \rho(t)\|_{L_p}\leq \exp(\int_0^t \|\nabla u\|_{L_\infty}\,d\tau)\int_0^t \|\nabla \rho\|_{L_p}\,d\tau$$
%which together with \eqref{eq:nrlpnew} implies \eqref{eq:twenrlp}.
\end{proof}

\medbreak
\begin{proposition}\label{prop:r2}
    Let $(\rho, v)$  be solution of  \eqref{eq:rh} and \eqref{eq:nskeqvc1}, and the initial data
    $\sqrt{\rho_0}v_0\in L_2,$ $\rho_0\in H^1 $ and there exist positive constant $c$ depend on $\rho_*,$ $\rho^*,$ $\nu$ satisfy
    $$\|\sqrt{\rho_0}v_0\|_{ L_2}+\|\rho_0-1\|_{L_2}+\|\nabla \rho_0\|_{L_2}\leq c .$$
    Additionally, the density $\rho$ satisfy
    \begin{equation}
        \label{eq:assr}
        0< \rho_*\leq \rho \leq \rho^* \andf \underset{t\in \R_+}{\sup}\|\nabla  \rho(t)\|^2_{L_4}\leq \delta \ll 1.
    \end{equation}
    Then, we have $$v\in L_{\infty}(\R_+;L_2(\R^d)) \andf \nabla v  \in L_{2}(\R_+;L_2(\R^d)),$$
    $$\rho \in L_{\infty}(\R_+;H^1(\R^d)) \andf \nabla \rho  \in L_{2}(\R_+;H^1(\R^d)).$$

\end{proposition}

\begin{proof}

Before showing the proof,  by using  the interpolation inequality:
\begin{equation}
    \label{eq:gnl4}\|z\|_{L_4}\lesssim \|z\|^{\theta}_{L_2}\|\nabla z\|^{1-\theta}_{L_2} \with \theta=1-\frac d4
\end{equation}
and Young inequality, we  present the following inequalities which will be  frequently used:
\begin{equation}
    \label{eq:usein}
    \begin{aligned}
    \|\nabla \rho\|^2_{L_4}\|v\|^2_{L_4}&\leq \|\nabla \rho\|^2_{L_4}\|v\|^{2\theta}_{L_2}\|\nabla v\|^{2(1-\theta)}_{L_2}\\
    &\leq \|\nabla \rho\|^{2\theta}_{L_4}\|v\|^{2\theta}_{L_2}\|\nabla v\|^{2(1-\theta)}_{L_2}\|\nabla \rho\|^{2(1-\theta)}_{L_4}\\
    &\leq C_{\rho_*} (\|\nabla \rho\|^2_{L_4}\| \sqrt{\mu(\rho)}\nabla v\|^2_{L_2}+\|\sqrt{\rho}v\|^2_{L_2}\|\nabla \rho\|^2_{L_4}).
\end{aligned}
\end{equation}
Now, in order to get the estimates of $\|\nabla \rho\|_{L_2(L_2)}$ and forbid the initial data $\|\rho_0\|_{L_2}$ small, 
we consider $a:=\rho-1$ instead of $\rho$ in the $L_2$ estimates. From \eqref{eq:rh}, the equation of $a$ reads as:
$$a_t-\nu\Delta a+\div((a+1)v)=0.$$
Multiplying with $a$ and integrating on $\R^d$ yields:
$$\frac 12\frac{d}{dt} \int_{\R^d} \abs{a}^2\,dx+\nu \int_{\R^d} \abs{\nabla a}^2\,dx=\int_{\R^d} (a+1) v\cdot \nabla a\,dx\leq \|\nabla a\|_{L_2}\|\sqrt{\rho}v\|_{L_2} \sqrt{\rho^*}.$$
On the other hand, testing \eqref{eq:rh} with $-\Delta \rho,$   we get
\begin{multline*}
     \frac 12\frac{d}{dt} \int_{\R^d} \abs{\nabla \rho}^2\,dx+\nu \int_{\R^d} \abs{\Delta  \rho}^2\,dx=\int_{\R^d} \div(\rho v)\cdot \Delta \rho\,dx\\
   \leq  \|\Delta \rho\|_{L_2}(\|\nabla \rho\|_{L_4} \|v\|_{L_4}+C_{\rho^*}\|\nabla v\|_{L_2})
\end{multline*}
By using Young inequality and note that $\nabla a=\nabla \rho$  we have
\begin{equation}
    \label{eq:esyt}
    \begin{aligned}
     \frac{d}{dt} \nu \int_{\R^d} \abs{ a}^2\,dx+\nu^2 \int_{\R^d} \abs{ \nabla\rho}^2\,dx &\leq \rho^*\|\sqrt{\rho}v\|^2_{L_2},\\
   \frac{d}{dt} \nu \int_{\R^d} \abs{\nabla \rho}^2\,dx+\nu^2 \int_{\R^d} \abs{\Delta  \rho}^2\,dx&\leq  \|\nabla \rho\|^2_{L_4} \|v\|^2_{L_4}+C_{\rho^*,\rho_*}  \|\sqrt{\mu(\rho)}\nabla v\|^2_{L_2}.
 \end{aligned}
\end{equation}

Multiplying \eqref{eq:nskeqvc1} with $v$ and integrating on $\R^d$  gives
\begin{equation}
    \label{eq:1}
    \frac 12\frac{d}{dt} \int_{\R^d} \rho \abs{v}^2\,dx+\frac{1}{2}\int_{\R^d} \mu(\rho)\abs{D(v)}^2\,dx=-\nu \int_{\R^d} \nabla \rho\cdot v\cdot \nabla v\,dx-\int_{\R^d} \nabla P\cdot v\,dx+\int_{\R^d} \phi\cdot v\,dx.
\end{equation}
Firstly, we present the term of pressure. Recalling that the definition of $v$ and $\div =0,$ we find that
$$-\int  \nabla P \cdot v\,dx=\int P\cdot \div v\,dx=\nu \int P\cdot \Delta \log \rho\,dx.$$
Therefore, we test the momentum equation of $(INSK)$
with  $\nu \nabla \log \rho$ to get
\begin{multline}
    \nu \int P \cdot\Delta  \log \rho\,dx=\nu\int \rho \d_t u\cdot \nabla \log \rho+\nu\int \rho u\cdot  \nabla u \cdot \nabla \log \rho\\
    -\nu\int \div(\mu(\rho)D(u))\cdot \nabla \log \rho\,dx+\nu\int \div(k(\rho)\nabla \rho\otimes \nabla \rho)\cdot \nabla \log \rho\,dx:=B_{i=1}^4.
\end{multline}
As $\div u=0,$ we have $B_1=0.$ Remembering that $u=v-\nu \nabla \log \rho$ and  the fact $$D(\nabla \log \rho)=2\nabla(\nabla \log \rho)$$ we get
\begin{align*}
    B_3&=-\nu\int \div(\mu(\rho)D(u))\cdot \nabla \log \rho\,dx\\
    &=\nu\int \mu(\rho)D(u) \cdot \nabla(\nabla \log \rho)\,dx\\
    &=\nu\int \mu(\rho)D(v) \cdot \nabla(\nabla \log \rho)\,dx-2\nu^2 \int\mu(\rho)\abs{\nabla(\nabla \log \rho)}^2 \,dx
\end{align*}
and
\begin{align*}
    B_2&=\nu\int u\cdot \nabla u\cdot \nabla \rho\,dx,\\
    &=\nu\int v\cdot \nabla v \cdot \nabla \rho\,dx-\nu^2\int v\cdot \nabla (\nabla \log \rho)\cdot \nabla \rho\,dx\\
    &\qquad\qquad-\nu^2 \int \nabla \log \rho\cdot \nabla v\cdot \nabla \rho\,dx+\nu^3 \int \nabla \log \rho\cdot \nabla(\nabla \log \rho)\cdot \nabla \rho\,dx.
\end{align*}
Moreover, for the last  term of $\phi$, we have
\begin{align*}
     \int 2\nu \div(\mu(\rho)\nabla(\nabla \log \rho))\cdot v\,dx&=-\int 2\nu \mu(\rho)\nabla(\nabla \log \rho)\cdot \nabla v\,dx\\
     &=-\int 2 \mu(\rho)\nabla(v-u)\cdot \nabla v\,dx\\
     &=\int 2 \mu(\rho)\nabla u\cdot \nabla v\,dx-2\|\sqrt{\mu(\rho)}\nabla v\|_{L_2}^2.
\end{align*}

Inserting the above results into \eqref{eq:1} and recalling the definition of $\phi,$ we obtain
\begin{small}
\begin{multline}
      \label{eq:2}
    \frac 12\frac{d}{dt} \int \rho \abs{v}^2\,dx+\frac{1}{2}\int \mu(\rho)\abs{D(v)}^2\,dx+2\int \mu(\rho)\abs{\nabla v}^2\,dx+2\nu^2\int \mu(\rho)\abs{\nabla(\nabla \log \rho)}^2\,dx\\
    =-\nu \int \nabla \rho\cdot v\cdot \nabla v\,dx
    -\nu^2 \int \nabla \log \rho\cdot \nabla v\cdot \nabla \rho\,dx+\nu^3 \int \nabla \log \rho\cdot \nabla(\nabla \log \rho)\cdot \nabla \rho\,dx\\
    +\nu\int \mu(\rho)D(v) \cdot \nabla(\nabla \log \rho)\,dx+\nu\int \div(k(\rho)\nabla \rho\otimes \nabla \rho)\cdot \nabla \log \rho\,dx\\
    -\int \div(k(\rho)\nabla \rho\otimes \nabla \rho)\cdot v\,dx-\nu\int \nabla \rho\cdot Dv\cdot v\,dx\\
    +2\nu\int \nabla \rho\cdot \nabla v\cdot v\,dx+2 \int \mu(\rho) \nabla u\cdot \nabla v\,dx:=\sum_{i=1}^9 A_i.
\end{multline}
\end{small}

Next, we bound the right hand side term by term. For the first term, taking advantage of Young inequality and \eqref{eq:usein}, one has
\begin{equation}
    \label{eq:esa1}
    \begin{aligned}
        A_1&= -\nu \int \nabla \rho \cdot v\cdot \nabla v\,dx\\
        &\leq \nu \|\nabla v\|_{L_2}\|\nabla \rho\|_{L_4}\|v\|_{L_4}\\
        &\leq \ep \| \sqrt{\mu(\rho)}\nabla v\|^2_{L_2}+C_{\rho_*,\ep}\nu^2 \|\nabla \rho\|^2_{L_4}\|v\|^2_{L_4}\\
        &\leq  \ep \| \sqrt{\mu(\rho)}\nabla v\|^2_{L_2}+ C_{\rho_*,\ep} \nu^2(\|\nabla \rho\|^2_{L_4}\| \sqrt{\mu(\rho)}\nabla v\|^2_{L_2}+\|\sqrt{\rho}v\|^2_{L_2}\|\nabla \rho\|^2_{L_4})
    \end{aligned}
\end{equation}

According to \eqref{eq:gnl4} and \eqref{eq:usein}, we obtain
\begin{equation}\label{eq:esa678}
    \begin{aligned}
   A_6&=-\int k(\rho)\nabla \rho\otimes \nabla \rho\cdot \nabla v\,dx,\\
    &\leq C_{\rho_*,\rho^*} \|\nabla v\|_{L_2}\|\nabla \rho\|^2_{L_4}\\
    &\leq \ep \|\sqrt{\mu(\rho)}\nabla v\|^2_{L_2}+C_{\ep, \rho_*,\rho^*}\|\nabla \rho\|^4_{L_4},\\
    A_7+A_8&=\nu\int \nabla \rho \cdot (2\nabla v-Dv)\cdot v\,dx\\
    &\lesssim \nu (\|\nabla v\|_{L_2}+\|D v\|_{L_2})\|\nabla \rho\|_{L_4}\|v\|_{L_4}\\
    & \lesssim \ep (\| \sqrt{\mu(\rho)}\nabla v\|^2_{L_2}+\| \sqrt{\mu(\rho)}Dv\|^2_{L_2})+C_{\rho_*,\ep}\nu^2\|v\|^2_{L_4}\|\nabla \rho\|^2_{L_4}\\
    & \lesssim \ep (\| \sqrt{\mu(\rho)}\nabla v\|^2_{L_2}+\| \sqrt{\mu(\rho)}Dv\|^2_{L_2})+C_{\rho_*,\ep}\nu^2 (\|\nabla \rho\|^2_{L_4}\| \sqrt{\mu(\rho)}\nabla v\|^2_{L_2}+\|\sqrt{\rho}v\|^2_{L_2}\|\nabla \rho\|^2_{L_4})
\end{aligned}
\end{equation}

Next, from Young inequality, we obtain
\begin{equation}
    \label{eq:esa2}
    \begin{aligned}
    A_{2}& \leq  \nu^2 1/\rho_*\|\nabla v\|_{L_2}\|\nabla \rho\|^2_{L_4}\\
   & \leq \ep \| \sqrt{\mu(\rho)}\nabla v\|^2_{L_2}+C_{\rho_*,\ep}\nu^4\|\nabla \rho\|^4_{L_4},
\end{aligned}
\end{equation}
\begin{equation}\label{eq:esa3}
    \begin{aligned}
        A_3&\leq \nu^3 1/\rho_*\|\nabla (\nabla \log \rho)\|_{L_2}\|\nabla \rho\|^2_{L_4}\\
        &\leq \ep \nu^2 \|\sqrt{\mu(\rho)}\nabla(\nabla \log \rho)\|^2_{L_2}+C_{\rho_*,\ep}\nu^4\|\nabla \rho\|^4_{L_4},
    \end{aligned}
\end{equation}
\begin{equation}
    \label{eq:esa4}
    A_4\leq \nu^2 \|\sqrt{\mu(\rho)}\nabla(\nabla \log \rho)\|^2_{L_2}+\frac{1}{4}\|\sqrt{\mu(\rho)}D(v)\|^2_{L_2},
\end{equation}
\begin{equation}
    \label{eq:esa5}
    \begin{aligned}
    A_5&=-\nu\int k(\rho)\nabla \rho \otimes \nabla \rho\cdot \nabla(\nabla \log \rho)\,dx\\
    &\leq C_{\rho^*} \nu\|\nabla (\nabla \log \rho)\|_{L_2}\|\nabla \rho\|^2_{L_4}\\
    &\leq \ep \nu^2 \|\sqrt{\mu(\rho)}\nabla(\nabla \log \rho)\|^2_{L_2}+C_{\rho^*,\rho_*,\ep}\|\nabla \rho\|^4_{L_4},
    \end{aligned}
\end{equation}
For rest term $A_9,$ we use H\"older inequality and   Young inequality again to get
\begin{equation}
    \label{eq:esa9}
    A_{9}\leq \ep \|\sqrt{\mu(\rho)}\nabla v\|^2_{L_2}+C_{\ep} \|\sqrt{\mu(\rho)}\nabla u\|^2_{L_2}.
\end{equation}
 Finally, thanks to \eqref{eq:nrlpnew}, \eqref{eq:gnl4} and Young inequality, we drive
\begin{align}
    \label{eq:esl2l4}
  &  \|\nabla \rho\|^{2}_{L_4}\lesssim  \|\nabla \rho\|^{2\theta}_{L_2}\|\nabla^2 \rho\|^{2(1-\theta)}_{L_2}\lesssim \|\nabla \rho\|^{2}_{L_2}+  \|\nabla^2 \rho\|^{2}_{L_2}\\
   \label{eq:esl4l4}
         &\|\nabla \rho\|^4_{L_4}\lesssim \|\nabla \rho\|^2_{L_4} \|\nabla \rho\|^{2}_{L_4}\lesssim \|\nabla \rho\|^2_{L_4}( \|\nabla \rho\|^{2}_{L_2}+ \|\nabla^2 \rho\|^{2}_{L_2})\cdotp
\end{align}
Putting \eqref{eq:esa1}, \eqref{eq:esa678},\eqref{eq:esa2},\eqref{eq:esa3},\eqref{eq:esa4},\eqref{eq:esa5},\eqref{eq:esa9},\eqref{eq:esl2l4},\eqref{eq:esl4l4}  and choosing small enough $\ep$, we deduce that
\begin{multline}
\label{eq:3}
     \frac 12\frac{d}{dt} \int \rho \abs{v}^2\,dx+\frac{1}{4}\int \mu(\rho)\abs{D(v)}^2\,dx+\frac 14\int \mu(\rho)\abs{\nabla v}^2\,dx+\frac 12\nu^2\int \mu(\rho)\abs{\nabla(\nabla \log \rho)}^2\,dx\\
     \lesssim C_{\rho_*,\ep} \nu^2\|\sqrt{\rho}v\|^2_{L_2}\|\nabla \rho\|^2_{L_4}
     +C_{\rho_*,\rho^*,\ep}(\nu^4+1)\|\nabla \rho\|^4_{L_4}+C_{\ep} \|\sqrt{\mu(\rho)}\nabla u\|^2_{L_2}\\
       \lesssim C_{\rho_*,\ep} \nu^2\|\sqrt{\rho}v\|^2_{L_2}(\|\nabla \rho\|^{2}_{L_2}+  \|\nabla^2 \rho\|^{2}_{L_2})
    +C_{\ep} \|\sqrt{\mu(\rho)}\nabla u\|^2_{L_2}\\
     +C_{\rho_*,\rho^*,\ep}(\nu^4+1)\|\nabla \rho\|^2_{L_4}( \|\nabla \rho\|^{2}_{L_2}+ \|\nabla^2 \rho\|^{2}_{L_2})
\end{multline}

Denote 
\begin{multline*}
    X(t):=\\
    \frac 12 \sup_{\tau\leq t} \int \rho(\tau) \abs{v(\tau)}^2\,dx+\frac 14\int_0^t\int \mu(\rho)\abs{\nabla v}^2\,dxd\tau+ \frac 12\nu^2\int_0^t\int \mu(\rho)\abs{\nabla(\nabla \log \rho)}^2\,dxd\tau
\end{multline*}
$$$$
$$Y(t):=\nu  \sup_{\tau\leq t}\int(\abs{a(\tau)}^2+\abs{\nabla\rho(\tau)}^2)\,dx+\nu^2 \int_0^t \int(\abs{\nabla\rho(\tau)}^2+\abs{\nabla^2\rho(\tau)}^2)\,dxd\tau$$

Integrating on $[0,t]$ of the time for \eqref{eq:esyt}, and \eqref{eq:3} we obtain
\begin{equation}
    \label{eq:xt}
    X(t)\leq C_{\rho_*,\ep} X(t)Y(t)+C_\ep \int_0^t   \|\sqrt{\mu(\rho)}\nabla u\|^2_{L_2}\,d\tau+C_{\rho_*,\rho^*,\ep}\frac{(\nu^4+1)}{\nu^2}\|\nabla \rho(t)\|^2_{L_4} Y(t)+X_0,
\end{equation}
and 
\begin{equation}
    \label{eq:yt}
    Y(t)\leq C_{\rho_*} X(t)Y(t)+ C_{\rho_*,\rho^*} X(t)(1+Y(t))+Y_0
\end{equation}
Multiplying \eqref{eq:yt} with $\frac{1}{2C_{\rho_*,\rho^*}}$ and plugging with \eqref{eq:xt}, if assuming that $\|\nabla \rho\|^2_{L_4}\ll 1 ,$ thus we see that
\begin{equation*}
    \frac 12 X(t)+\frac{1}{2C_{\rho_*,\rho^*}} Y(t)\lesssim C_{\rho_*,\rho^*} X(t)Y(t) +C_{\ep} \int_0^t   \|\sqrt{\mu(\rho)}\nabla u\|^2_{L_2}\,d\tau+X_0+Y_0
\end{equation*}
which along with \eqref{eq:enes}  implies for all $t\in \R_+$
$$X(t)+Y(t)\leq C(X_0+Y_0+(X(t)+Y(t))^2)$$
where $C$ only depend on $\rho^*,$ $\rho_*,$ and $\nu.$
Hence, if 
\begin{equation}
    \label{eq:assxy0}
    4C(X_0+Y_0)<1
\end{equation}
we have
\begin{equation}
    \label{eq:esxy}
    X(t)+Y(t)\leq 2(X_0+Y_0).
\end{equation}

\end{proof}
Having  the above estimates of density and $v$, now,  we come back to the estimates of $u.$
\begin{proposition}\label{prop0ad2} Let $1<p,q<2$ be such that $1/p+1/q=3/2,$  and $s$ and $m$ being the conjugate exponents of $q$ and
$p,$ respectively.
Let $(\rho,u)$ be a smooth solution of  $(INSK)$ on $\R_+\times\R^2$
 and initial  density satisfying
\begin{equation}\label{eq:smallarho1}\norm{\rho_0-1}_{L_{\infty}\cap L_{2} }, \norm{\rho_0}_{\H^1\cap \W^1_4}, \|1/\sqrt{\rho_0}\nabla \rho_0\|_{L_2},\|\sqrt{\rho_0}u_0\|_{L_2} \leq c\ll1 \andf  \nabla \rho_0 \in L_{m}.
\end{equation}
Additionally, assuming 
\begin{equation}
    \label{eq:inilip}
    \int^\infty_0 \|\nabla u\|_{L_\infty}\leq \frac{1}{2}.
\end{equation}
  Then, it holds that
\begin{equation}
    \label{eq:esnr}
    \nabla \rho\in L_{s,1}(\R_+; L_{m}) \andf \| \nabla \rho\|_{L_{s,1}(L_{m})}\leq C_0
\end{equation}
where $C_0$ only depends on the initial data.
For velocity, we have for all $1<p,q<2$ with $1/p+1/q=3/2,$  
\begin{multline*}
\norm{u}_{L_{\infty}(\R_+;
\B^{-1+2/p}_{p,1})}+\norm{u_{t}, \nabla^{2} u,\nabla P}_{L_{q,1}(\R_+;L_{p})}\\+\norm{u}_{L_{s,1}(\R_+;L_{m})}
\leq C (\norm{u_{0}}_{\B^{-1+2/p}_{p,1}}
+cC_0),\end{multline*}
for a constant $C.$ 
Furthermore, we have
  \begin{equation*}
    \norm{\t{u}}_{L_{q,1}(\R_+;L_{p})} \leq C (\norm{u_{0}}_{\B^{-1+2/p}_{p,1}}
+cC_0)
\end{equation*}
and 
 \begin{equation*}\norm{u}_{L_2(\R_+;L_{\infty})} \leq C (\norm{u_{0}}_{\B^{-1+2/p}_{p,1}}
+cC_0)\cdotp\end{equation*}
\end{proposition}
\begin{proof} 
Putting together the energy balance, \eqref{eq:smallarho1}, \eqref{eq:inilip}, Proposition \ref{prop:r1} and \ref{prop:r2}
 yields the following results of $\rho$ in $\R^d$:
\begin{itemize}
    \item $\underset{t\in \R_+}{\sup} \|\rho(t)-1\|_{L_\infty}\leq c, \underset{t\in \R_+}{\sup} \|\nabla \rho(t)\|_{L_4(\R^d)}\leq c.$
    \item $\nabla \rho \in L_\infty(\R_+; L_2(\R^2))\cap L_2(\R_+; H^1(\R^2)),$ $\nabla \rho \in L_\infty(\R_+; L_{m}(\R^2))$ and such that 
    \begin{align}
      \label{eq:nal2h1}  \|\nabla \rho \|_{L_\infty(\R_+; L_2(\R^2))\cap L_2(\R_+; H^1(\R^2))}\lesssim c\\
     \label{eq:lilm}   \|\nabla \rho\|_{L_\infty(\R_+; L_{m}(\R^2))}\lesssim \|\nabla \rho_0\|_{L_{m}(\R^2)}.
    \end{align}
   
\end{itemize}
Hence, as $m>2,$ the Gagliardo–Nirenberg inequality gives
\begin{equation*}
    \|\nabla \rho\|_{L_2(L_{m}(\R^d))}\lesssim \|\nabla \rho\|^{\theta}_{L_2(L_{2}(\R^2))} \|\nabla^2 \rho\|^{1-\theta}_{L_2(L_{2}(\R^2))}\with \theta=\frac{2}{m}. 
\end{equation*}
Thus, the interpolation $L_{s,1}(\R^2)=(L_\infty(\R^2),L_2(\R^2))_{2/s,1}$ together\eqref{eq:lilm} implies \eqref{eq:esnr}.

\medbreak

%and \begin{equation}\label{edu1}
 %\rho \t{u}=\mu \Delta u -\nabla P \with \t{u}:= u_t+u\cdot\nabla u.
%\end{equation}
%Because $\div u=0,$ applying $\div$ to \eqref{s4e1} yields: 
%\begin{equation}\label{eqnp}
 %\nabla P=-\nabla (\Delta)^{-1}\div [(\rho-1)u_{t}+\rho u\cdot \nabla u],
%\end{equation}
%which implies 
%\begin{equation}\label{s4e2}
%u_{t}-\mu \Delta u=-\p [(\rho-1)u_{t}+\rho u\cdot \nabla u ],
%\end{equation}
%where $\p={\rm Id}-\nabla (\Delta)^{-1}\div$ is the Helmholtz projector on divergence free vector field.
Looking at \eqref{eq:uh} as a Stokes   equation with source term,   Proposition \ref{propregularity} (taking $2-2/q=-1+2/p$) gives us
 \begin{multline}\label{esb2da}
     \norm{u}_{L_{\infty}(\R_+;
\B^{-1+2/p}_{p,1})}+\norm{u_{t}, \nabla^{2} u,\nabla P}_{L_{q,1}(\R_+;L_{p})}+\norm{u}_{L_{s,1}(\R_+;L_{m})}\\
\leq C\bigl(\norm{u_{0}}_{\B^{-1+2/p}_{p,1}}+\norm{(\rho-1)u_{t}+\rho u\cdot \nabla u}_{L_{q,1}(\R_+;L_{p})}+\norm{g}_{L_{q,1}(\R_+;L_{p})}\bigr)\cdotp
 \end{multline}
We then have by H\"older inequality,
$$\displaylines{\quad
\norm{(\rho-1)u_{t}+\rho u\cdot \nabla u}_{L_{q,1}(\R_+;L_{p})}\leq
%&\,\norm{ [(\rho-1)u_{t}+\rho u\cdot \nabla u ]}_{L_{2}(0,T\times\R^{2})}\\\lesssim &\,
\norm{\rho-1}_{L_{\infty}(\R_+\times\R^{2})}\norm{u_{t}}_{L_{q,1}(\R_+;L_{p})}\hfill\cr\hfill+\norm{\rho}_{L_{\infty}(\R_+\times\R^{2})}\norm{u\cdot \nabla u}_{L_{q,1}(\R_+;L_{p})}+\norm{g}_{L_{q,1}(\R_+;L_{p})}.\quad}
$$
As $c$ is small enough in \eqref{eq:smallarho1}, then the first part on the right-hand side can be absorbed by the left-hand side of \eqref{esb2da}.
 For  the nonlinear term, by H\"older inequality and embedding $$\W^1_p(\R^2)\hookrightarrow L_{p^*}(\R^2) \with \frac{1}{p^*}=\frac 1p-\frac 12,$$
  we  have
$$\norm{u\cdot \nabla u}_{L_{q,1}(\R_+;L_{p})}\leq \norm{u}_{L_{\infty}(\R_+;L_{2})}\norm{\nabla^2 u}_{L_{q,1}(\R_+;L_{p})}.$$
Then, the energy estimates \eqref{eq:enes} implies
$$\norm{u\cdot \nabla u}_{L_{q,1}(\R_+;L_{p})}\leq c\norm{\nabla^2 u}_{L_{q,1}(\R_+;L_{p})}.$$
Hence, the nonlinear term can also be absorbed by the left hand side of \eqref{esb2da}.

From \eqref{eq:muu}, \eqref{eq:smallarho1} and the energy estimates \eqref{eq:enes} drive
\begin{equation*}
    \|\div((\mu(\rho)-\mu(1))D(u))\|_{L_{q,1}(\R_+;L_{p})}\lesssim c\|\nabla^2 u\|_{L_{q,1}(\R_+;L_{p})}.
\end{equation*}
Recalling $c\ll 1$, then  $\|\div((\mu(\rho)-\mu(1))D(u))\|_{L_{q,1}(\R_+;L_{p})}$ can be absorbed by the left hand side.
Now, we begin to deal with the capillarity term.
Note that $1/p=1/2+1/m, 1/q=1/2+1/s,$ thus, in light of 
$$\|z\|^2_{L_{4}(\R_+\times\R^2)}\leq \|z\|_{L_{\infty}(\R_+;L_{2})}\|\nabla z\|_{L_{2}(\R_+;L_{2})}$$ and \eqref{eq:nal2h1}, one get
\begin{align*}
    \|\nabla \rho\cdot \nabla^2 \rho\|_{L_{q,1}(\R_+;L_{p})}\leq \|\nabla \rho\|_{L_{s,1}(\R_+;L_{m})}\|\nabla^2 \rho\|_{L_{2}(\R_+\times\R^2)} \leq c\|\nabla \rho\|_{L_{s,1}(\R_+;L_{m})};\\
     \|\Delta \rho\cdot \nabla \rho\|_{L_{q,1}(\R_+;L_{p})}\leq \|\nabla \rho\|_{L_{s,1}(\R_+;L_{m})}\|\nabla^2 \rho\|_{L_{2}(\R_+\times\R^2)} \leq c\|\nabla \rho\|_{L_{s,1}(\R_+;L_{m})};\\
      \||\nabla \rho|^2 \nabla \rho\|_{L_{q,1}(\R_+;L_{p})}\leq \|\nabla \rho\|_{L_{s,1}(\R_+;L_{m})}\|\nabla \rho\|^2_{L_{4}(\R_+\times\R^2)} \leq c\|\nabla \rho\|_{L_{s,1}(\R_+;L_{m})}.\\
\end{align*}
Whereas $k$ is smooth function and \eqref{eq:krho} leads to
$$\div(k(\rho)\nabla \rho \otimes \nabla \rho)\|_{L_{q,1}(\R_+;L_{p})}\lesssim c\|\nabla \rho\|_{L_{q_1,1}(\R_+;L_{p_1})}.$$

Putting all the estimates together, we conclude that
\begin{multline*}
\norm{u}_{L_{\infty}(\R_+;
\B^{-1+2/p}_{p,1})}+\norm{u_{t},\nabla^{2} u,\nabla P}_{L_{q,1}(\R_+;L_{p})}\\+\norm{u}_{L_{s,1}(\R_+;L_{m})}
\leq C (\norm{u_{0}}_{\B^{-1+2/p}_{p,1}}+c\|\nabla \rho\|_{L_{s,1}(\R^2;L_{m})}).\end{multline*}
Finally, by similar statement of \eqref{eq:dotu2} and \eqref{eq:uLinfty} completes the proof of the proposition.\end{proof}
\begin{remark}
   Owing to Proposition \ref{prop:r1} and \ref{prop:r2}, we can obtain the Proposition \ref{prop0ad2} which improve Proposition \ref{prop0d2} in the two dimension case by moving the smallness assumption of $\|u_0\|_{\B^{-1+2/p}_{p,1}(\R^2)}.$ 
\end{remark}
\appendix

\begin{small}	 

\end{small}
\bigbreak\bigbreak
\noindent\textsc{Istituto per le Applicazioni del Calcolo ''Mauro Picone'' }
\par\nopagebreak
E-mail addresses: shan.wang@u-pec.fr, \quad s.wang@iac.cnr.it

\end{document}